\documentclass[11pt,a4paper]{amsart}
\usepackage{subfiles}
\usepackage{hyperref}

\newenvironment{changemargin}[2]{%
\begin{list}{}{%
\setlength{\topsep}{0pt}%
\setlength{\leftmargin}{#1}%
\setlength{\rightmargin}{#2}%
\setlength{\listparindent}{\parindent}%
\setlength{\itemindent}{\parindent}%
\setlength{\parsep}{\parskip}%
}%
\item[]}{\end{list}}

\hypersetup{urlcolor=blue, citecolor=blue, linkcolor=blue, colorlinks=true}

\usepackage{enumitem}
\usepackage{amsmath}
\usepackage{array}
\usepackage{amssymb}
\usepackage{mathtools}
\usepackage{amsfonts}
\usepackage{microtype}
\usepackage{multicol}
\usepackage{subcaption}
\usepackage{xcolor}
\usepackage{stackengine,scalerel}
\usepackage{tikz}
\usepackage{csquotes}
\usepackage{tikz-cd}
\usepackage{amsthm}
\usepackage{tcolorbox}
\usepackage[capitalise]{cleveref}
\usepackage[utf8]{inputenc}
\usepackage[english]{babel}

\makeatletter
\newcommand{\leqnomode}{\tagsleft@true\let\veqno\@@leqno}
\newcommand{\reqnomode}{\tagsleft@false\let\veqno\@@eqno}
\makeatother

\newcolumntype{L}{>{\scriptstyle}l}
\newcolumntype{C}{>{\scriptstyle}c}
\newcolumntype{R}{>{\scriptstyle}r}
\newenvironment{mysubarray}{%
  \scriptstyle
  \setlength\arraycolsep{0pt}%
  \setlength\extrarowheight{-1ex}
  
  \begin{array}{RCL}}{\end{array}}

\DeclareMathAlphabet{\mathpzc}{OT1}{pzc}{m}{it}

\usepackage{newunicodechar}
\usepackage[style=alphabetic,backend=biber,maxnames=5,maxalphanames=5,doi=true,isbn=false,url=false]{biblatex}
\bibliography{literature.bib}
\renewbibmacro{in:}{}
\renewbibmacro*{issue+date}{%
  \ifboolexpr{not test {\iffieldundef{year}} or not test {\iffieldundef{issue}}}
    {\printtext[parens]{%
       \iffieldundef{issue}
         {\usebibmacro{date}}
         {\printfield{issue}%
          \setunit*{\addspace}%
          \usebibmacro{date}}}}
    {}%
  \newunit}
\let\oldtocsection=\tocsection
\let\oldtocsubsection=\tocsubsection
\let\oldtocsubsubsection=\tocsubsubsection
\renewcommand{\tocsection}[2]{\hspace{0em}\oldtocsection{#1}{#2}}
\renewcommand{\tocsubsection}[2]{\hspace{1em}\oldtocsubsection{#1}{#2}}
\renewcommand{\tocsubsubsection}[2]{\hspace{2em}\oldtocsubsubsection{#1}{#2}}
\newtheorem{bigthm}{Theorem}
\newtheorem{bigcor}[bigthm]{Corollary}

\newtheorem{thm}{Theorem}[section]
\newtheorem{lem}[thm]{Lemma}
\newtheorem{prop}[thm]{Proposition}
\newtheorem{cor}[thm]{Corollary}
\newtheorem{question}[thm]{Question}
\newtheorem{problem}[thm]{Problem}

\theoremstyle{definition}
\newtheorem{dfn}[thm]{Definition}
\newtheorem*{definition}{Definition}

\theoremstyle{remark}
\newtheorem{example}[thm]{Example}

\newtheorem{rem}[thm]{Remark}

\newtheorem{construction}[thm]{Construction}

\renewcommand{\max}{{\mathrm{max}}}
\renewcommand{\min}{{\mathrm{min}}}

\setenumerate[1]{leftmargin=*,labelindent=0pt,label=(\roman*),}
\newcommand{\pref}[2]{\hyperref[#2]{#1 \ref*{#2}}}
\usepackage{todonotes}

\newcommand{\id}{\ensuremath{\operatorname{id}}}

\newcommand{\Spin}{\ensuremath{\operatorname{Spin}}}

\newcommand{\oH}{H}

\newcommand{\diff}{\ensuremath{\operatorname{Diff}^{\scaleobj{0.8}{+}}}}
\newcommand{\Diff}{\ensuremath{\operatorname{Diff}}}
\newcommand{\blockdiff}{\ensuremath{\widetilde{\operatorname{Diff}}^{{\scaleobj{0.8}{+}}}}}

\newcommand{\diffs}{\ensuremath{\operatorname{Diff}^{\scaleobj{0.8}{\Spin}}}}
\newcommand{\maps}{\ensuremath{\operatorname{maps}}}

\newcommand{\sign}{\ensuremath{\operatorname{sign}}}
\newcommand{\sing}{\ensuremath{\operatorname{Sing}}}
\newcommand{\colim}{\ensuremath{\operatorname{colim}}}
\newcommand{\GL}{\ensuremath{\operatorname{GL}}}
\newcommand{\ahat}{\ensuremath{{\hat\calA}}}

\newcommand{\bo}{\ensuremath{\operatorname{BO}}}
\newcommand{\bg}{\ensuremath{\operatorname{BG}}}
\newcommand{\go}{\ensuremath{\operatorname{G/O}}}
\newcommand{\thom}{\ensuremath{\operatorname{Th}}}
\newcommand{\too}{\longrightarrow}
\newcommand{\rk}{\mathrm{rank}}
\newcommand{\pr}{\ensuremath{\operatorname{pr}}}

\newcommand{\hur}{\mathrm{hur}}
\newcommand{\fib}{\mathrm{Fib}^h}
\newcommand{\blockfib}[1]{\widetilde{\mathrm{Fib}}^h_{#1}}
\newcommand{\congarrow}{\overset{\cong}\longrightarrow}
\newcommand{\embeds}{\hookrightarrow}
\newcommand{\actson}{\curvearrowright}
\DeclarePairedDelimiter{\scpr}{\langle}{\rangle}
\DeclarePairedDelimiter\ceil{\lceil}{\rceil}
\DeclarePairedDelimiter\floor{\lfloor}{\rfloor}

\newcommand{\haut}{\ensuremath{\operatorname{hAut}}}
\newcommand{\blockhaut}{\ensuremath{\widetilde{\operatorname{hAut}}}}

\newcommand{\calR}{\mathcal{R}}
\newcommand{\calA}{\mathcal{A}}

\newcommand{\calN}{\mathcal{N}}

\newcommand{\calS}{\mathcal{S}}
\newcommand{\calE}{\mathcal{E}}
\newcommand{\calO}{\mathcal{O}}
\newcommand{\calL}{\mathcal{L}}
\newcommand{\calI}{\mathcal{I}}

\newcommand{\bbN}{\mathbb{N}}

\newcommand{\bbC}{\mathbb{C}}

\newcommand{\bbZ}{\mathbb{Z}}
\newcommand{\bbQ}{\mathbb{Q}}
\newcommand{\bfQ}{\mathbb{Q}}

\newcommand{\bbR}{\mathbb{R}}
\newcommand{\Reals}{\mathbb{R}}

\newcommand{\op}[1]{\mathbb{OP}^{#1}}
\newcommand{\hp}[1]{\mathbb{HP}^{#1}}
\newcommand{\cp}[1]{\mathbb{CP}^{#1}}

\newcommand{\bdiff}{B\!\diff}

\newcommand{\ch}{{\mathrm{ch}}}
\newcommand{\ph}{{\mathrm{ph}}}
\newcommand{\ko}{{\mathrm{KO}}}
\newcommand{\pt}{{\mathrm{pt}}}

\newcommand{\psec}{{\mathrm{Sec}>0}}
\newcommand{\nnsec}{{\mathrm{Sec}\ge0}}
\newcommand{\prc}{{\mathrm{Ric}>0}}

\newcommand{\psc}{{\mathrm{scal}>0}}

\renewcommand{\hom}{\mathrm{Hom}}

\begin{document}
\author[Georg Frenck]{Georg Frenck\\ Appendix with Jens Reinhold}
\email{\href{mailto:georg.frenck@math.uni-augsburg.de}{georg.frenck@math.uni-augsburg.de}}
\email{\href{mailto:math@frenck.net}{math@frenck.net}}
\address{{Institut für Mathematik}, {Universität Augsburg},{Universitätsstraße 14}, {Augsburg}, {86159}, {Bavaria}, {Germany}}

\keywords{Fiber bundles, characteristic classes, positive curvature, block bundles, surgery theory, rational cobordism classes.}

\subjclass[2010]{53C21,55R40, 57R20, 57R22, 58D17, 58D05.}

\thanks{G.~F.~is supported by the DFG (German Research Foundation) -- 281869850 (RTG 2229) and through the Priority Programme “Geometry at Infinity” (SPP 2026, ZE 1123/2-2).}

\title[manifold bundles over spheres and positive curvature]{Characteristic numbers of manifold bundles over spheres and positive curvature via block bundles}

\begin{abstract} 
	Given a simply connected manifold $M$, we completely determine which rational monomial Pontryagin numbers are attained by fiber homotopy trivial $M$-bundles over the $k$-sphere, provided that $k$ is small compared to the dimension of $M$. Furthermore we study the vector space of rational cobordism classes represented by such bundles. We give upper and lower bounds on its dimension and we construct manifolds for which these bounds are attained. The proof is based on the classical approach to studying diffeomorphism groups via block bundles and surgery theory and we make use of ideas developed by Krannich--Kupers--Randal-Williams.
	
	As an application, we show the existence of elements of infinite order in the homotopy groups of the spaces of positive Ricci and positive sectional curvature, provided that $M$ is $\Spin$, has a non-trivial rational Pontryagin class and admits such a metric. This is done by constructing $M$-bundles over spheres with non-vanishing $\ahat$-genus. Furthermore, we give a vanishing theorem for generalised Morita--Miller--Mumford classes for fiber homotopy trivial bundles over spheres.

	In the appendix co-authored by Jens Reinhold it is (partially) determined which classes of the rational oriented cobordism ring contain an element that fibers over a sphere of a given dimension.
\end{abstract}

\maketitle

%!TEX root = main.tex

\section{Introduction}
\noindent Let $M$ be a closed oriented manifold of dimension $d\ge5$. In this appendix we investigate the following question: Given an integer $k\ge1$ and a universal characteristic class $c\in H^{d+k}(\bo;\bbQ)$\footnote{Since $H^\ast(\bo;\bbQ)$ is concentrated in degrees divisible by $4$, we restrict to the case $d+k=4m$ throughout this article.}, does there exist a fiber bundle $M\to E\to S^k$ such that  $\scpr{c(E),[E]}\not=0$? If it does, then $c$ is called \emph{spherical for $M$}. Furthermore, $c$ is called \emph{$h$-spherical for $M$}, if $E$ can be chosen to be fiber homotopy trivial, that is $E$ comes equipped with a homotopy equivalence $E\simeq M\times S^k$ over $S^k$. Obviously, $h$-spherical classes are spherical. The following is our main result.

\begin{bigthm}\label{thm:main}
	Let $M^d$ be a simply connected, closed manifold and let $k$ be such that $1\le k\le \min(\frac{d+2}3, \frac{d-3}2)$ and $d+k=4m$.
	\begin{enumerate}
		\item A monomial $p =p_{i_1}\cup\cdots\cup p_{i_n}\not =p_{m}$ in universal rational Pontryagin classes of total degree ${d+k}$ is $h$-spherical for $M$ \emph{if and only if} there exists an $\ell\in\{1,\dots,n\}$ such that 
		\[p_{i_1}(TM)\cup\cdots\cup\widehat{p_{i_\ell}(TM)}\cup\cdots\cup p_{i_n}(TM)\not=0.\]
		\item Let $M$ admit a nontrivial rational Pontryagin class and let $p_i(TM)$ have the lowest degree among these. Then there exists a fiber bundle $E\to S^k$ such that 
		\[\scpr{p_i(TE)\cup p_{m-i}(TE),[E]} \not=0\not=\scpr{p_{m}(TE),[E]}\]
		are the \emph{only} nonzero monomial Pontryagin numbers of $E$. In particular the following are equivalent: 
		\begin{enumerate}
			\item The class $p_m$ is $h$-spherical for $M$ 
			\item The class $p_m$ is spherical for $M$
			\item $M$ admits some nontrivial rational Pontryagin class.
		\end{enumerate}
%		Furthermore, if $i\not=d/4$ then $E$ admits a cross-section with trivial normal bundle.
		\item For every $\ell>n\ge3$ the class ${p_1^{n+\ell}}$ is spherical but not $h$-spherical for $\cp{n}\times\cp{2\ell}$.	
	\end{enumerate} 
\end{bigthm}

\noindent We remark that (iii) of the above theorem follows from \pref{Theorem}{thm:main} (i) and \pref{Proposition}{prop:p_1^n}. The latter goes back to a joint work with Jens Reinhold, which now forms the jointly written appendix to this article.

\begin{rem}\label{rem:main}
\begin{enumerate}
	\item It is known that no characteristic class $c\in H^{d+k}(\bo;\bbQ)$ is spherical for any $M$ if ${k}>2d$ (cf.~\cite[Lemma 2.3]{Wiemeler}). This implies the necessity for a bound on $k$, even though the one we give in \pref{Theorem}{thm:main} might not be optimal. This bound can be improved depending on the connectivity of $M$: Let $M$ be a $d$-dimensional, $\ell$-connected manifold with $d\ge5$ and $\ell\ge1$. We say that $k\ge1$ is in the \emph{unblocking range for $M$} if one of the following is satisfied
	\begin{enumerate}
		\item $k\le \min(\frac{d+2}3, \frac{d-3}2)$
		\item $d$ is even and $k\le\min(d-4, 2\ell-1)$.
		\item $d$ is odd, $(k-1)$ is not divisible by $4$ and $k\le\min(d-6,2\ell-1)$.
	\end{enumerate}
	\pref{Theorem}{thm:main} holds for all $k$ in the unblocking range for $M$. Note, that any of the above conditions implicitely enforces $d\ge5$ if we want $k\ge1$.
	\item If all rational Pontryagin classes of $M$ vanish, then again no characteristic class $c\in H^{d+k}(\bo;\bbQ)$ is spherical by \cite[Proposition 1.9]{HankeSchickSteimle}\footnote{Both \cite[Lemma 2.3]{Wiemeler} and \cite[Proposition 1.9]{HankeSchickSteimle} are only stated for the $\ahat$-class, but the given proofs apply to any $c\in H^{d+k}(\bo;\bbQ)$.}, see also \pref{Theorem}{MainThmAppendix}, part (i)). In particular, this proves $(b)\Rightarrow(c)$ in \pref{Theorem}{thm:main}, (ii): If $p_m$ is spherical, then some rational Pontryagin class of $M$ must be nonzero. Note that $(a)\Rightarrow (b)$ is trivial and $(c)\Rightarrow (a)$ follows from the first half of \pref{Theorem}{thm:main}, (ii).
\end{enumerate}
\end{rem}

\noindent Next, let $\fib_{M,k}\subset \Omega_{d+k}\otimes\bbQ$ denote the set of classes represented by a fiber homotopy trivial $M$-bundles $E\to S^{k}$. Note that $\fib_{M,k}$ is a linear subspace since it is given by the image of the transfer homomorphism
\[\pi_k\left(\frac{\haut(M)}{\diff(M)}\right)\otimes\bbQ\too\Omega_{d+k}\otimes\bbQ,\]
where $\haut(M)/\diff(M)$ denotes the classifying space for fiber homotopy trivial $M$-bundles and the above map is given by sending $S^k\to\haut(M)/\diff(M)$ to the bundle classified by it. We will now give estimates for the dimension of $\fib_{M,k}$. For this, let $i_{\min}$ be the minimum positive integer $i$ such that $p_i(TM)\not=0$ and let $n_{\max}$ be the maximum integer $n\ge1$ such that $p_{i_{\min}}(TM)^n\not=0$.

\begin{bigthm}\label{main:lowerbound}
	Let $M$ be  simply  connected  and let $k\ge1$ be in the unblocking range for $M$ sucht that $d+k=4m$. Then, for every $1\le n\le n_{\max}$ there exists a fiber homotopy trivial $M$-bundle $E_n\to S^k$ with the property that for $\ell\ge1$ we have 
	\[\scpr{p_{i_{\min}}(TE_n)^\ell\cup p_{m - \ell\cdot i_{\min}}(TE_n),[E_n]}\not=0\quad \iff\quad n=\ell.\]
\end{bigthm}
\noindent Since $\Omega_*\otimes\bbQ$ is classified by Pontryagin-numbers we get a lower bound on $\dim\fib_{M,k}$ which we prove to be attained for certain manifolds.
\begin{bigcor}\label{cor:lowerbound}
	\begin{enumerate}
		\item We have $\dim\fib(M,k)\ge n_{\max}$. 
		\item If all Pontryagin classes of $M$ are contained in the truncated polynomial $\bbQ$-algebra generated by $p_{i_{min}}(TM)$, then $\dim\fib(M,k)= n_{\max}$.
	\end{enumerate}
\end{bigcor}

\begin{proof}
	\begin{enumerate}
		\item Consider the linear homomorphism $\Omega_{d+k}\too\bbQ^{n_{\max}}$ given by 
	\[[X]\mapsto \Bigl(\scpr{p_{i_{\min}}(TX)^\ell\cup p_{\frac{d+k}{4} - \ell\cdot i_{\min}}(TX),[X]}\Bigr)_{\ell=1,\dots,n_{\max}}.\]
%	which is injective since $\Omega_*\otimes\bbQ$ is classified by Pontryagin-numbers. 
If $\calE$ denotes the vector space spanned by $E_1,\dots E_{n_{\max}}$ from \pref{Theorem}{main:lowerbound}, the composition
\[\calE\to\fib(M,k)\to\Omega_{d+k}\otimes\bbQ\to\bbQ^{n_{\max}}\]
is given by a diagonal matrix with nonzero entries on the diagonal by \pref{Theorem}{main:lowerbound}. Hence, the first map is forced to be injective.
		\item is proven in \pref{Section}{sec:proofnmax} below.\qedhere
	\end{enumerate}
\end{proof}

\begin{example}
	The prototypical examples of manifolds for which this lower bound from \pref{Corollary}{cor:lowerbound} is attained are $\cp{a}, \hp{b}$ and $\op2$ for $a,b\ge2$. If $k_a\equiv2a\mod 4$ and $k_a\le\min(\frac{2a+2}3,a-\frac32)$, then 
	\[\dim\fib_{\cp{a},k_a} = n_\max(\cp{a}) = \floor{\frac a2}.\]
	Analogously, for $k_b, k_c$ divisible by $4$ and $k_b\le\min(\frac{4b+2}{3}, 2b-\frac32)$ (or $k_b\le4$) and $k_c\le 12$, we obtain:
	\[\dim\fib_{\hp{b},k_b} = b,\qquad \dim\fib_{\op{2},k_c} = 2.\]
\end{example}

\noindent In order to describe the upper bound, recall that
\[H^{*}(\bo(d);\bbQ)=	\bbQ[p_1,\dots, p_{\floor{\frac{d}2}}].\]
Let $p(n)$ be the number of partitions of $n\in\bbN$ into sums of positive natural numbers and let us fix $m\coloneqq \frac{d+k}4$. The assumption of $k$ being in the unblocking range guarantees that $k\le d-2$ which implies that $4m=k+d\le2d-2\le\deg(p_{\floor{\frac{d}2}})$ and hence we have $\dim H^{4m}(\bo(d);\bbQ) = p(m)$. Furthermore, for $\ell\in \bbN$ we define $p(n,\ell)$ to be the number of partitions of $n$ into natural numbers $\le \ell$. Note that $p(n,n)=p(n)$, $p(n,0) = 0$, $p(n,1)=1$ and $p(n,2) = 1 + \floor{n/2}$. Furthermore $p(n,\ell) = \calO(n^{\ell-1})$. 

We have the following observation considering an upper bound on $\dim\fib_{M,k}$: If $i_1,\dots, i_r$ is such that $\sum i_j=4m$ and $i_j < k/4$ for all $j$, then we  have \(\scpr{{p_{i_1}(TE)\cdots p_{i_r}(TE)},[E]} = 0\) for every fiber homotopy trivial $M$-bundle $E\to S^k$ by the following argument: By our assumption on $(i_j)$, the degree of $p_{i_1}(TM)\cdots\widehat {p_{i_\ell}(TM)}\cdots p_{i_r}(TM)$ equals $4m - 4i_j>d$ and hence this class vanishes since the corresponding cohomology group of $M$ vanishes. The claim follows from \pref{Theorem}{thm:main}, (i). We get the following upper bound:
\begin{bigthm}\label{main:upperbound}
	Let $M$ be  simply  connected  and let $k\ge1$ be in  the unblocking range for $M$. Then for $4m=d+k$, we have $\dim\fib_{M,k}\le p(m) - p(m, m-\ceil{\frac{d+1}{4}}) -1$. There exist a simply connected manifold $M$ in dimensions $d\equiv2,3\;(4)$ for which equality holds.
\end{bigthm}

\noindent The upper bound is an immediate consequence of the above observation together with the fact that $\sigma(E)=0$. The main difficulty of \pref{Theorem}{main:upperbound} lies in proving sharpness. In order to do so, we construct a manifold $M$ and for every $I=(i_1,\dots, i_s)$ with $s\ge2$, $\sum i_j=m$ and $i_j\ge(4m-d)/4$ for some $j$, we construct a fiber homotopy trivial $M$-bundle $E_I\to S^k$ such that $\scpr{p_I(TE_I),[E_I]}$ and $\scpr{p_m(TE_I),[E_I]}$ are the \emph{only} nontrivial monomial Pontryagin numbers of $E_I$ (\pref{Lemma}{lem:blocklarge} together with \pref{Lemma}{lem:unblocking}). As $\Omega_{d+k}\otimes\bbQ$ is classified by Pontryagin numbers, all $(E_I)_{I \text{ as above}}$ are linearly independent in $\Omega_{d+k}\otimes\bbQ$. It follows that 
\[\dim\fib_{M,k}\ge \underbrace{|\left\{(i_1,\dots, i_s)\colon s\ge2,\ \sum i_j=m\text{ and }i_j\ge(4m-d)/4\right\}|}_{ = p(m)-p(m-\ceil{\frac{d+1}{4}}) -1}.\]

\begin{rem}\label{rem:upperbound}
	If $d\equiv 0\;(4)$ there exists a non-connected manifold $M$ where every component is simply connected such that equality holds.
\end{rem}

\subsection{Sphericity of the \texorpdfstring{$\ahat$}{A-hat}-genus}

Among all characteristic numbers, there are 2 of particular interest: The signature and the $\ahat$-genus. It is well-known that the signature is multiplicative in fiber bundles with simply connected base, so the Hirzebruch $\calL$-class is not spherical. The analogous statement is known to be wrong for the $\ahat$-genus: In \cite[Proposition 1.10]{HankeSchickSteimle}, Hanke--Schick--Steimle constructed a fiber bundle $E\to S^k$ with $\ahat(E)\not=0$. Their construction however \enquote{is based on abstract existence results in differential topology [and] does not yield an explicit description of the diffeomorphism type of the [...] manifold} {[}loc.~cit.~p.~3{]}. This has been resolved by Krannich--Kupers--Randal-Williams in \cite{KKRW}, where they constructed a fiber bundle $\hp2\to E\to S^4$ with $\ahat(E)\not=0$. Employing part (ii) of \pref{Theorem}{thm:main} together with the computational result \cite[Lemma 2.5]{a-hat-bundles} we can go far beyond their result.

\begin{prop}\label{prop:a-hat-multip}
	Let $M$ be  simply  connected  and let $k\ge1$ be in  the unblocking range for $M$ such that $d+k=4m$. Then the following are equivalent:
	\begin{enumerate}
		\item $\ahat_{m}$ is h-spherical for $M$.
		\item $\ahat_{m}$ is spherical for $M$.
		\item $M$ admits a nontrivial rational Pontryagin class.
	\end{enumerate}
\end{prop}

\noindent The $\hat\calA$-genus is particularly interesting in the presence of a $\Spin$-structure as it obstructs the existence of positive scalar curvature in this case. Hence, \pref{Proposition}{prop:a-hat-multip} implies that there exist fiber bundles that are homotopy equivalent to the trivial bundle and where base and fiber both admit positive sectional curvature metrics, whereas the total space does not even support positive scalar curvature. 

This can also be applied to the study of spaces of Riemannian metrics with lower curvature bounds. Let $M$ be closed and let $\calR_\psc(M)$ denote the space of Riemannian metrics on $M$ of positive scalar curvature, equipped with the Whitney $C^\infty$-topology. Furthermore, let $\calR_C(M)$ be a $\diff(M)$-space which admits a $\diff(M)$-equivariant map to $\calR_\psc(M)$ and let $\diff(M,D)\subset\diff(M)$ denote the subgroup of those diffeomorphisms that fix an embedded disk $D\subset M$ point-wise.

\begin{bigthm}\label{main:curvature}
	Let $M^d$ be closed, simply connected $\Spin$-manifold that has at least one non-vanishing rational Pontryagin class and let $g\in\calR_C(M)$. Let $k\ge1$ be such that $(d+k)$ is divisible by $4$ and $k$ is in the unblocking range for $M$. Then for $k\ge1$ the image of the map
	\begin{center}
	\begin{tikzpicture}
		\node (0) at (0,0.7){$\pi_{k-1}(\diff(M,D))$};
		\node (1) at (5,0.7){$\pi_{k-1}(\calR_C(M))$};
		
		\draw[->] (0) to (1);
	\end{tikzpicture}
	\end{center}
	induced by the orbit map $f\mapsto f_*g$ contains an element of infinite order.
\end{bigthm}

\noindent For $k\ge3$, the\pref{Theorem}{main:curvature} was the original motivation and the main theorem of the first version of the present article. For readers solely interested in the proof of this theorem we recommend the original version which is considerably more focused and available at \href{https://arxiv.org/abs/2104.10595v1}{2104.10595v1}.

\begin{example}\label{ex:manifolds}
\begin{enumerate}
	\item The class of manifolds to which this theorem is applicable contains $\cp{2n+1}$, $\hp n$, $\op 2$, as well as iterated products and connected sums of these with \emph{arbitrary} $\Spin$-manifolds.
	\item The most interesting examples of spaces $\calR_C(M)$ are the ones of positive or nonnegative sectional curvature and positive Ricci curvature metrics. \pref{Theorem}{main:curvature} implies, that for $M$ and $k$ as above we have
\begin{align*}
	\pi_{k-1}(\calR_\prc&(M))\otimes\bbQ\not=0\quad
	\pi_{k-1}(\calR_\psec(M))\otimes\bbQ\not=0\\
	&\text{and }\pi_{k-1}(\calR_\nnsec(M))\otimes\bbQ\not=0.
\end{align*}
provided the respective spaces are non-empty. The case of non-negative sectional curvature follows from a Ricci-flow argument, see \cite[Proposition 3.3]{a-hat-bundles}.
\end{enumerate}
\end{example}

%\begin{rem}[State of the art concerning $\calR_\psec(M)$]
%	To the best of the author's knowledge, the following is all that is known about the homotopy type of $\calR_\psec(M)$:
%	\begin{enumerate}
%		\item Kreck--Stolz have shown in \cite{kreckstolz}, that there exists a manifold $M$ such that $\calR_{\psec}(M)$ is not connected. Their result remains true fore the quotient $\calR_\psec(M)/\diff(M)$, so those components do not originate from the orbit map.
%		\item Crowley--Schick--Steimle have shown in \cite{crowleyschicksteimle} that for every $\Spin$-manifold $M$ of dimension $d\ge6$, the image of the orbit map 
%		\[\pi_{k-1}(\diff(M))\to\pi_{k-1}(\calR_\psec(M))\] contains $\bbZ/2$ as a subgroup provided $d+k\equiv1,2\;(8)$. This extends earlier results of Hitchin \cite{hitchin_spinors} and Crowley--Schick \cite{crowleyschick}.
%		\item Krannich--Kupers--Randal-Williams have proven the special case $M=\hp2$, $k=4$ of \pref{Theorem}{main:curvature} in \cite{KKRW} utilizing the same construction as the one in the present paper.
%	\end{enumerate}	 
%\end{rem}

\noindent According to \cite{ziller} the only known examples of positively curved manifolds in dimensions $4k+3$ for $k\ge2$ are spheres. Also, all $7$-dimensional examples have finite fourth cohomology (cf. \cite{eschenburg, goette_pk, berger_space}). Therefore, the answer to the following question appears to be unknown. A positive answer would yield the first example of a closed manifold that admits infinitely many pairwise non-isotopic metrics of positive sectional curvature.

\begin{question}
	Is there a positively curved manifold of dimension $4k+3$, $k\ge1$ with a non-vanishing rational Pontryagin class?
\end{question} 

\noindent For positive Ricci and nonnegative sectional curvature, lots of examples for such manifolds are known, see \pref{Example}{ex:manifolds}.

\subsection{A vanishing result for Morita--Miller--Mumford classes}

We denote by $\diff(M)$ the group of orientation preserving diffeomorphisms of $M$ with $B\!\diff(M)$ the associated classifying space. For any $M$-bundle $\pi\colon E\to B$ with with structure group $\diff(M)$ there is a map
\[H^{4m}(\bo(d);\bbQ) \to H^{{4m}-d}(B;\bbQ)\]
sending a characteristic class $c\in H^{4m}(\bo(d);\bbQ)$ to $\kappa_c(E)\coloneqq\pi_!(c(T_\pi E))$ where $\pi_!\colon H^*(E)\to H^{*-d}(B)$ is the Gysin-homomorphism and $T_\pi E$ is the vertical tangent bundle of $\pi$. $\kappa_c(E)$ is called the \emph{generalized Miller-Morita-Mumford class} or simply \emph{$\kappa$-class} associated to $c$. For the universal $M$-bundle $\pi_M\colon E_M\to B\!\diff(M)$ we hence get universal $\kappa$-classes 
\[\kappa_c\coloneqq \kappa_c(E_M)\in H^{{4m}-d}(B\!\diff(M);\bbQ)\]
Let $\haut(M)/\diff(M)$ be the classifying space for fiber homotopy trivial $M$-bundles. We define the maps $\Psi_{M,m}$ and $\Psi^h_{M,m}$ as follows
\begin{center}
\begin{tikzpicture}
\node(0) at (0,1.2) {$H^{4m}(\bo(d);\bbQ)$};
\node(1) at (7,1.2) {$H^{4m-d}(\bdiff(M);\bbQ)\cong \hom(H_{{4m}-d}(\bdiff(M));\bbQ)$};
\node(2) at (9,0) {$\hom(\pi_{{4m}-d}(\bdiff(M));\bbQ)$};
\node(3) at (9,-1.2) {$\hom(\pi_{{4m}-d}(\haut(M)/\diff(M));\bbQ)$};
\draw[->] (0) to node[auto]{$(\pi_M)_!$} (1);
\draw[->] (9,0.9) to node[auto]{$\hur^*$} (9,0.3);
\draw[->] (9,-0.3) to (9,-0.9);
\draw[->,dashed](0) to node[below, yshift=-5pt, xshift=30pt]{$\Psi_{M,m}$} (2.west);
\draw[->,dashed](0) to node[below]{$\Psi^h_{M,m}$} (3.west);
\end{tikzpicture}
\end{center}
where $\hur$ denotes the Hurewicz-homomorphism. Obviously, $\ker(\Psi_{M,m})\subset \ker(\Psi_{M,m}^h)$. Since the signature of a fiber bundle over $S^k$ vanishes, The Hirzebruch $\calL$-class lies in $\ker(\Psi_{M,m})$ for all $M$ and $m$. The maps $\Psi^h_{M,m}$ and $\Psi_{M,m}$ are both given by $c\mapsto (f\mapsto\scpr{f^*\kappa_c,[S^{4m-d}]})$ and we have
\begin{align*}
	\scpr{f^*\kappa_{c},[S^{4m-d}]} &= \scpr{\kappa_c(E), [S^{4m-d}]}= \scpr{\pi_! c(T_{\pi}^sE),[S^{4m-d}]}\\
		&=\scpr{c(T_{\pi}^sE),[E]}=\scpr{c(\pi^*TS^{4m-d}\oplus T_{\pi}^sE),[E]} = \scpr{c(TE),[E]}
\end{align*}
since $TS^{4m-d}$ is stably parallelizable and hence $c(TS^{4m-d}) = 1$. With this, \pref{Theorem}{thm:main} translates to the following.
\begin{bigthm}\label{main:kappa}
	Let $M$ be simply connected and let $\tau\colon M\to \bo(d)$ be a classifying map for $TM$. Furthermore, let $k\ge1$ be in the unblocking range for $M$ and let $4m=d+k$. Then
	\begin{enumerate}
		\item If $p=p_{i_1}\cdots  p_{i_r}\not=p_m$ is a product of universal Pontryagin classes of degree $4m$, then 
			\[p\in\ker(\Psi_{M,m}^h) \iff \tau^*(p_{i_1}\cdots\widehat{p_{i_\ell}}\cdots p_{i_r})=0\text{ for all } \ell\]
		\item The following are equivalent:
		\begin{enumerate}
			\item$p_m\in\ker(\Psi_{M,m}^h)$
			\item$p_m\in\ker(\Psi_{M,m})$
			\item$\tau^*p_i=0$ for all $i\in\{1,\dots,\floor{\frac  d4}\}$.
		\end{enumerate}
		\item For $\ell>n\ge3$ and $M=\cp{2\ell}\times\cp{n}$, we have 
			\[p_1^{n+\ell}\in \ker(\Psi_{M,(n+\ell)}^h)\setminus\ker(\Psi_{M,(n+\ell)}).\]
	\end{enumerate}
\end{bigthm}
%For $1\le i\le m$, let $P_i\coloneqq\_\cup p_i\colon H^{4m-i}(\bo(d);\bbQ)\too H^{4m}(\bo(d);\bbQ)$. We obtain the following vanishing result for Morita--Miller--Mumford  classes  as a corollary. 
%\begin{bigcor}
%For all $1\le i<m$ such that $0=p_i(TM)=\tau^*p_i$, we have $P_i(\ker(\tau^*))\subset\ker(\Psi_{M,m}^h)$.
%\end{bigcor}

\noindent Furthermore, we obtain the following bounds on the dimension of $\ker(\Psi_{M,m}^h)$: For $n_{\max}$ and $p(m,\ell)$ as above, we have:
\[p(m, m-\ceil{\frac{d+1}{4}}) +1\le \dim\ker(\Psi_{M,m}^h)\le p(m)-n_{\max}\] 
For both bounds there exist manifolds such that the bounds are attained. 
%The following vanishing result of $\kappa$-classes for fiber homotopy trivial $M$-bundles over spheres is an immediate consequence of \pref{Theorem}{main:kappa}.
%
%\begin{bigcor}
%	Let $P_i\coloneqq p_i\cup\_\colon H^{4(m-i)}(BO(d);\bbQ)\to H^{4m}(BO(d);\bbQ)$ and let $(\tau^*)_{4(m-i)}\colon H^{4(m-i)}(BO(d);\bbQ)\to H^{4m-i}(M;\bbQ)$ be the map induced by $\tau$ in degree $4(m-i)$. Then
%	\[\bigoplus_{i<m\colon p_i(TM)=0}P_i(\ker(\tau^*)_{4(m-i)})\subset \ker\Psi_{M,m}^h.\]
%\end{bigcor}
%
%\begin{proof}
%	If $p_I\in\ker(\tau^*)_{4(m-i)}$ and $p_i(TM)=0$, then for all $j\in I$ we have $\tau^*(p_{I\setminus\{j\}}\cup p_i)=0$ and by \pref{Theorem}{main:kappa} (i), $P_i(p_I)=p_i\cup p_I\in\ker(\Psi_{M,m}^h)$.
%\end{proof}

\subsection{Rationally fibering a cobordism class over a sphere}
Given an integer $k\ge1$, it is a classical problem that goes back to Conner--Floyd \cite{ConnorFloyd} to decide which cobordism class contains a manifold that fibers over $S^k$. This has been studied in detail for $k\le4$ \cite{Burdick, Neumann, Kahn1, Kahn2}. However, the classical approach relied on identifications like $S^2\cong\cp1$ or $S^4\cong \hp1$ and does not seem to work for larger $k$. \pref{Theorem}{main:upperbound} can be rephrased to state that any rational cobordism class in degrees $4m\ge16$ fibers over $S^4$, provided that the signature vanishes. In the Appendix, written with Jens Reinhold, we consider the question for bigger values of $k$ building on the methods developed in the paper. We show that in a given dimension $d\ge 32$ every (rational) cobordism class in the kernel of the signature homomorphism fibers over $S^k$ for every $k\le 8$. We also obtain results for $k\ge9$, see \pref{Theorem}{MainThmAppendix}.

\subsection{Outline of the argument and obstructions to unblocking of block bundles}

In \cite{KKRW}, Krannich--Kupers--Randal-Williams have proven that the class $\hat\calA_3\in H^{12}(\bo(8);\bbQ)$ is $h$-spherical for $\hp2$. It turns out that their construction delivers an excellent blueprint for our generalization. Since \cite{KKRW} is written rather densely, we decided to give a more detailed account of their argument in \pref{Section}{sec:prelim} before we go on to proving our main results. Let us give an outline of the construction first.

Instead of constructing an actual fiber one constructs a so-called block bundle (we recall the notion of block bundles and block diffeomorphisms in \pref{Section}{sec:prelim}). The advantage of working with block bundles is that the $k$-th homotopy group $\pi_k(\haut(M)/\blockdiff(M))$ of the classifying space for fiber homotopy trivial block bundles is isomorphic to the structure set $\calS_\partial(D^k\times M)$ from surgery theory. Since we assumed that $\dim(M)\ge5$, the latter is accessible through the surgery exact sequence
\begin{equation*}
	L_{k+d+1}(\bbZ\pi_1M)\too\calS_\partial(D^k\times M) \too \calN_\partial(D^k\times M)\overset{\sigma}{\too} L_{k+d}(\bbZ\pi_1M)
\end{equation*}
where $\calN_\partial$ denotes the set of normal invariants. We are interested in the case, where $M$ is simply connected and $(d+k)$ is divisible by $4$, so the $L$-groups are given by $0$ on the left and by $\bbZ$ on the right. Hence, in order to construct an $M$-block bundle it suffices to construct a normal invariant $\eta$ with $\sigma(\eta)=0$. It turns out, that the set of normal invariants is (rationally) isomorphic to the reduced real $K$-theory of $S^k\wedge M_+$ which allows one to construct a normal invariant and hence a (fiber homotopy trivial) block bundle with prescribed Pontryagin classes. Fiber homotopy trivial block bundles over $S^k$ can (rationally) be given the structure of an actual fiber bundle if $k$ is in the unblocking range for $M$. This follows from a classical result of Burghelea--Lashof (cf. \cite{BurgheleaLashof}) and Morlet's Lemma of disjunction together with work of Krannich and Randal-Williams \cite{RW_UpperRange,Krannich_NEW}. 

The main work in the present article lies in choosing appropriate normal invariants. This way we can ensure that certain Pontryagin classes and numbers of the total space of the corresponding block bundle are zero or nonzero. 

One further observation is, that the above construction of block bundles works regardless of the dimension $k$ of the base sphere. The same cannot be true for actual fiber bundles over spheres, since the tangent bundle of the total space is stably isomorphic to its vertical tangent bundle which is a vector bundle of rank $d$. Therefore all Pontryagin classes of degree $*>2d$ vanish. This has been previously observed in \cite{erw_block}, where the authors go on to construct an $\hp 2$-block bundle over $S^{12}$ with $p_5\not=0$ which for the above reason cannot be \enquote{unblocked}. Utilizing the above construction of block bundles we can study this obstruction to \enquote{unblocking} more systematically. Let $I=(i_1,\dots,i_s)$ be such that $|I|\coloneqq i_1+\dots+i_s<m$ and let $M$ be a manifold such that $p_I(TM) = p_{i_1}\cdot\hdots\cdot p_{i_s}(TM)\not=0$. If $I=(0)$ assume instead that $M$ has some non-vanishing rational Pontryagin class. Let
\[\widetilde{P_{M,I}}\colon\pi_{4m-d}(B\blockdiff(M))\otimes\bbQ\too \bbQ\]
be the map sending a block bundle $E$ classified by $f$ to the Pontryagin number $\scpr{p_I(TE)\cup p_{m-|I|}(TE),[E]}$. The following corollary follows from the proof of \pref{Theorem}{thm:main} in \pref{Section}{sec:prescribing}.

\begin{bigcor}\label{main:unblock}
	\begin{enumerate}
		\item The map $\widetilde{P_{M,I}}$ is surjective.
		\item If $m>\frac{d+2|I|}2$, then the following composition is trivial.
	\[\pi_{4m-d}(B\!\diff(M))\otimes\bbQ \too \pi_{4m-d}(B\blockdiff(M))\otimes\bbQ\overset{\widetilde{P_{M,I}}}\too \bbQ\]
		\item For $1\le n\le n_{\max}$ (cf. \pref{Equation}{align:imin}) and $m>\frac{d+2n\cdot i_{\min}}2$ there exists an $n$-dimensional subspace $\calN\subset \pi_{4m-d}(B\blockdiff(M))\otimes\bbQ$ of block-bundles that do not admit the structure of actual fiber bundles.
	\end{enumerate}
\end{bigcor}

\begin{rem}
	The same is true for fiber homotopy trivial bundles, i.e. the same corollary holds if $B\!\diff(M)$ and $B\blockdiff(M)$ are replaced by $\haut(M)/\diff(M)$ and $\haut(M)/\blockdiff(M)$. 
\end{rem}

\begin{proof}[Proof of \pref{Corollary}{main:unblock}]
\begin{enumerate}
	\item This follows from \pref{Lemma}{lem:realize-p}.
	\item By assumption $m-|I|>\frac{d+2|I|}2 - |I| = \frac d2$ and if $E\to S^{4m-d}$ is a fiber bundle, then $ p_{m-|I|}(TE) =  p_{m-|I|}(T_\pi E)=0$ since $T_\pi E$ is a vector bundle of rank $d$ and therefore its highest possible Pontryagin class is $p_{d/2}$.
	\item By \pref{Lemma}{lem:pmin}, there exist block bundles $E_j$ for $1\le j\le n$ with $\scpr{p_{i_\min}(TE)^j\cup p_{m-j\cdot i_\min}(TE),[E]}\not=0$ which yield linearly independent classes of $\Omega_{4m}\otimes\bbQ$ and hence in $\pi_{4m-d}(B\blockdiff(M))\otimes\bbQ$. Then $m-j\cdot i_\min > \frac{d+2n\cdot i_{\min}}2-n\cdot i_\min = \frac d2$. By the same argument as in (ii), these block bundles do not admit the structure of actual fiber bundles.\qedhere
\end{enumerate}
\end{proof}

\subsection*{Acknowledgements} I would like to thank the anonymous referee for a very thorough reading and for catching various errors and inaccuracies. Further thanks goes to Jens Reinhold for agreeing to jointly write the appendix to this article which grew out of a joint project. I also thank Johannes Ebert for a stimulating discussion about $\kappa$-classes as well as Manuel Krannich and Wolfgang Steimle for their interest and various comments on a draft of this article. Furthermore I thank Marek Kaluba for his interest in the computational aspects from the proof of \pref{Theorem}{MainThmAppendix}, as well as Bernhard Hanke and Jost-Hinrich Eschenburg for valuable remarks.

\section{Preliminaries}\label{sec:prelim}

\noindent Let $M$ be a closed oriented manifold of dimension $d$ and let $\diff(M)$ denote the group of orientation preserving diffeomorphisms of $M$. We denote by $B\!\diff(M)$ the classifying space for fiber bundles $E\to B$ with structure group $\diff(M)$.

\subsection{Block diffeomorphisms}
In this subsection we give a short overview of block bundles and diffeomorphisms and we explain how to compare them to honest fiber bundles and diffeomorphisms. For $p\ge0$ let $\Delta^p$ denote the standard topological $p$-simplex.\footnote{We choose to  follow \cite{erw_block} for this, even though there are more recent expositions on block diffeomorphisms like \cite{Krannich_NEW} or \cite{berglundmadsen} since the latter do not cover block bundles.}

\begin{dfn}
	 A \emph{block diffeomorphism} of $\Delta^p\times M$ is a diffeomorphism of $\Delta^p\times M$ that for each face $\sigma\subset\Delta^p$ restricts to a diffeomorphism of $\sigma\times M$. 
\end{dfn}

\noindent The set of all block diffeomorphisms forms a semisimplicial group denoted by $\blockdiff_{\bullet}(M)$ whose $p$-simplices are the block diffeomorphisms of $\Delta^p\times M$. The space $\blockdiff(M)$ of block diffeomorphisms is defined as the geometric realization of $\blockdiff_\bullet(M)$ and the associated classifying space is denoted by $B\blockdiff(M)$. This space classifies block bundles. Let us recall the definition of a block bundle over a simplicial complex.

\begin{dfn}[{\cite[Definition 2.4]{erw_block}}]
	Let $K$ be a simplicial complex and let $p\colon E\to |K|$ be continuous. A \emph{block chart} for $E$ over a simplex $\sigma\subset K$ is a homeomorphism $h_\sigma\colon p^{-1}(\sigma)\to \sigma\times M$ which for every face $\tau\subset \sigma$ restricts to a homeomorphism $p^{-1}(\tau)\to \tau\times M$. A \emph{block atlas} is a set $\calA$ of block charts, at least one over each simplex of $K$, such that transition functions are block diffeomorphisms. $E$ is called a block bundle if it admits a block atlas.
\end{dfn}

\noindent By \cite[Proposition 3.2]{erw_block} a block bundle $\pi\colon E\to B$ has a stable analogue of the vertical tangent bundle, i.e. there exists a stable vector bundle $T_\pi^sE\to E$ which is stably isomorphic to the vertical tangent bundle $T_\pi E$ provided that $E$ is an actual fiber bundle. If furthermore $B$ is a manifold, the total space $E$ is again a manifold and there is a stable isomorphism $T_{\pi}^sE \oplus \pi^*TB \cong_{\mathrm{st}} TE$ (cf. \cite[Lemma 3.3]{erw_block}). 

Next, let us consider the semisimplicial subgroup $\diff_\bullet(M)$ of those block diffeomorphisms that commute with the projection $\Delta^p\times M\to \Delta^p$. This gives precisely the $p$-simplices of the singular semisimplicial group $\sing_\bullet\diff(M)$. We have an inclusion $\sing_\bullet\diff(M)\subset\blockdiff_\bullet(M)$ and since the geometric realization of $\sing_\bullet(X)$ is homotopy equivalent to $X$ for any space $X$ (\cite[pp. 8]{handbook_at}), we get an induced map
\[B\!\diff(M)\too B\blockdiff(M).\]
\noindent Let $\haut(M)$ denote the group-like topological monoid of (orientation preserving) homotopy equivalences of $M$ with classifying space $B\!\haut(M)$. Again, let $\blockhaut(M)$ be the realization of the semisimplicial group of block homotopy equivalences defined analogously to block diffeomorphisms and let $B\blockhaut(M)$ be the corresponding classifying space. By \cite[Thm 6.1]{dold_partitions} $\haut(M)$ and $\blockhaut(M)$ are homotopy equivalent. Consider the following maps induced by inclusions:
\[B\blockdiff(M) \to B\blockhaut(M)\simeq B\!\haut(M)\qquad B\!\diff(M)\to B\!\haut(M)\]
and let $\haut(M)/\blockdiff(M)$ and $\haut(M)/\diff(M)$ denote the respective homotopy fibers. Note that $\haut(M)/\diff(M)$ (resp. $\haut(M)/\blockdiff(M)$) classifies $M$-bundles (resp. $M$-block bundles) together with a fiberwise (resp. blockwise) homotopy equivalence to the the product $M$-bundle, i.e. \emph{fiber homotopy trivial} $M$-bundles (resp. \emph{blockwise homotopy trivial} $M$-block bundles). We make the following definition: A characteristic class $c\in H^{d+k}(\bo(d);\bbQ)$ is called \emph{block-spherical} (resp. \emph{block-$h$-spherical}), if there exists an $M$-block bundle $E\to S^k$ (resp. a blockwise homotopy trivial one) with $\scpr{c(TE),[E]} \not=0$. We have the following implications:
\begin{center}
\begin{tikzpicture}
	\node (0) at (0,0) {$c$ is $h$-spherical};
	\node (1) at (4.5,-.6) {$c$ is block-$h$-spherical};
	\node (2) at (4.5,.6) {$c$ is spherical};
	\node (3) at (9,0) {$c$ is block-spherical};
	
	\draw[-{Implies},double, bend right=10] (0.south east) to (1.west);
	\draw[-{Implies},double, bend left=10] (0.north east) to (2.west);
	\draw[-{Implies},double, bend right=10] (1.east) to (3.south west);
	\draw[-{Implies},double, bend left=10] (2.east) to (3.north west);
\end{tikzpicture}
\end{center}

\noindent We have the following comparison result which follows from \cite{BurgheleaLashof}. 
\begin{lem}\label{lem:blr}
	If $k\le\min(\frac{d+2}{3},\frac{d-3}{2})$ then the natural map
	\[\pi_k\left(\frac{\haut(M)}{\diff(M)}\right)\left[\frac12\right] \too \pi_k\left(\frac{\haut(M)}{\blockdiff(M)}\right)\left[\frac12\right]\]
	is surjective.
\end{lem} 
\begin{rem}
Since the homotopy fiber of the map
\[\haut(M)/\diff(M)\to\haut(M)/\blockdiff(M)\] is given by $\blockdiff(M)/\diff(M)$, the same surjectivity result also holds for 
\[B\diff(M)\to B\blockdiff(M).\]
\end{rem}
\noindent Before we dive into the proof, let us recall the  stable range: For a manifold $M$, possibly with boundary, let $C(M)\coloneqq\{f\colon M\times [0,1]\to M\times [0,1]\text{ diffeomorphism }\colon f|_{M\times\{0\}}=\id\}$ be the space of \emph{pseudoisotopies}. There is a canonical map $e\colon C(M)\to C(M\times I)$ and we define
\[\phi(d)\coloneqq \max\{q\in\bbN\colon e \text{ is $q$-connected for all } M\text{  with } \dim(M)\ge d\}\]
which is called the \emph{pseudoisotopy stable range}. By a classical theorem of Igusa, $\phi(d)\ge\min\left(\frac{d-4}3, \frac{d-7}2\right)$ \cite{igusa_stability}.
\begin{proof}[Proof of \pref{Lemma}{lem:blr}]
	 For $\omega\in\bbN$, let $\haut(M)_{2,\omega}$ be the $\omega$-th stage Moore-Postnikov tower of $\haut(M)$ localized at $2$. By \cite[Theorem C7]{BurgheleaLashof}, there is a map $\haut(M)_{2,\omega}\to\blockdiff(M)/\diff(M)$ such that the projection $q\colon\blockdiff(D^d)\to \blockdiff(D^d)/\diff(D^d)$ factors through this map, provided that $\omega\le\phi(d)+1$. Passing to quotients, we obtain a split 
	 \[\rho\colon\haut(M)_{2,\omega}/\diff(M)\to\blockdiff(M)/\diff(M)\] of the map induced by inclusion. Therefore, we get a split of the long exact sequence 
	 \[\dots\to\pi_{\ell}\left(\frac{\blockdiff(M)}{\diff(M)}\right)\left[\frac12\right]\to \pi_{\ell}\left(\frac{\haut(M)}{\diff(M)}\right)\left[\frac12\right]
	 \to\pi_{\ell}\left(\frac{\haut(M)}{\blockdiff(M)}\right)\left[\frac12\right]\to\dots\] 
	as long as $\ell\le \phi(d)+1$. So the left-hand map is split-injective and the map one further to the left is surjective, provided that $\ell+1\le\phi(d)+2 = \min(\frac{d+2}{3},\frac{d-3}{2})$ which implies the statement of the theorem. 
\end{proof}

\noindent Therefore, an element of $\pi_k(\haut(M)/\blockdiff(M))\otimes\bbQ$ yields an $M$-bundle $E\to S^k$ that is fiber homotopy trivial, provided that the dimension of $M$ is high enough. The advantage of working with $\haut(M)/\blockdiff(M)$ instead of $\haut(M)/\diff(M)$ stems from the fact, that the former is accessible through surgery theory as we will review in the \pref{Section}{sec:surgery}.

\begin{rem}\label{rem:morlet}
	Another approach to compare $B\!\diff(M)$ and $B\blockdiff(M)$ is by using Morlet's lemma of disjunction as in \cite[Lemma]{KKRW}. Consider the following diagram of (homotopy) fibrations
	\begin{center}
		\begin{tikzpicture}
			\node(0) at (0,0) {$B\blockdiff_\partial(D^{d})$};
			\node(1) at (0,1.3) {$B\!\diff_\partial(D^{d})$};
			\node(2) at (4,0) {$B\blockdiff(M)$};
			\node(3) at (4,1.3) {$B\!\diff(M)$};
			\node(4) at (0,2.6) {$\frac{\blockdiff_\partial(D^{d})}{\diff_\partial(D^{d})}$};
			\node(5) at (4,2.6) {$\frac{\blockdiff(M)}{\diff(M)}$};
			
			\draw[->] (1) to (0); 			
			\draw[->] (0) to (2);			
			\draw[->] (1) to (3);			
			\draw[->] (3) to (2);		
			\draw[->] (4) to (5);			
			\draw[->] (5) to (3);			
			\draw[->] (4) to (1);
		\end{tikzpicture}
	\end{center}
	If $M$ is $\ell$-connected with $\ell\le d-4$, then the induced map on homotopy fibers is $(2\ell-2)$-connected by Morlet's lemma of disjunction (cf. \cite[Corollary 3.2 on page 29]{blr}). Now $\pi_k(B\blockdiff_\partial(D^{d}))\cong \pi_0(\diff_\partial(D^{d+k-1}))$ is isomorphic to the finite group of exotic spheres in dimension $(d+k)$. If $d$ is even, then $B\!\diff_\partial(D^{d})$ is rationally $(d-5)$-connected by \cite[Theorem 4.1]{RW_UpperRange}. On the other hand, if $d$ is odd and $k\le d-7$ is not divisible by $4$, then $\pi_{k}(B\!\diff_\partial(D^{d}))\otimes\bbQ$ is trivial by \cite[Corollary B]{Krannich_NEW}. Both \cite{RW_UpperRange} and \cite{Krannich_NEW} are generalizations of a classical result due to Farrell--Hsiang \cite{FarrellHsiang}.
	
	Therefore, in both these cases $\pi_k(\blockdiff(D^d)/\diff(D^d))\otimes \bbQ$ is trivial and the same is true for $\pi_k(\blockdiff(M)/\diff(M))\otimes\bbQ$, provided $k\le 2\ell-2$. This implies that the map 
	\[\pi_{k+1}(B\!\diff(M))\otimes\bbQ\to \pi_{k+1}(B\blockdiff(M))\otimes\bbQ\]
	is surjective which also holds for the induced map 
	\[\pi_{k+1}\left(\frac{\haut(M)}{\diff(M)}\right)\otimes\bbQ \too \pi_{k+1}\left(\frac{\haut(M)}{\blockdiff(M)}\right)\otimes\bbQ\]
	by the five-lemma.
\end{rem}

\noindent We summarize this discussion about unblocking in the following definition and lemma.

\begin{definition}
	Let $M$ be $\ell$-connected for some $\ell\le d-4$. We say that $k\in\bbN$ is \emph{in the unblocking range for $M$}, if one of the following holds.
	\begin{enumerate}
		\item $k\le\min(\frac{d+2}{3},\frac{d-3}{2})$.
		\item $d$ is even and $k\le\min(d-4, 2\ell-1)$.
		\item $d$ is odd, $(k-1)$ is not divisible by $4$ and $k\le\min(d-6,2\ell-1)$.
	\end{enumerate}
\end{definition}

\begin{lem}\label{lem:unblocking}
	Let $M$ be simply connected and let $k$ be in the unblocking range for $M$. Then the following maps are both surjective
	\begin{align*}
		\pi_k\left(\frac{\haut(M)}{\diff(M)}\right)\otimes\bbQ &\too \pi_k\left(\frac{\haut(M)}{\blockdiff(M)}\right)\otimes\bbQ\\
		\pi_k\Bigl(B\!\diff(M)\Bigr)\otimes\bbQ&\too \pi_k\Bigl(B\blockdiff(M)\Bigr)\otimes\bbQ.
	\end{align*}
\end{lem}

\begin{rem}\label{rem:unblocking}
	Even though this \pref{Lemma}{lem:unblocking} only states surjectivity, this still is enough to completely determine ($h$-)sphericity: Every fiber homotopy trivial bundle also is also a blockwise homotopy trivial block bundle and defined sphericity and $h$-sphericity of those. Therefore, a characteristic class $c\in H^{k+d}(\bo(d);\bbQ)$ is ($h$-)spherical for $M$ if and only if it is block-($h$-)-spherical for $M$, provided that $k$ is in the unblocking range for $M$.
\end{rem}
t

\subsection{Surgery theory}\label{sec:surgery}
Let $X$ be a simply connected manifold of dimension at least $5$ with boundary $\partial X$. The $\emph{structure set}$ $\calS(X,\partial X)$ of $(X,\partial X)$ (sometimes written as $\calS_\partial(X)$) is defined to be the set of equivalence classes of tuples $(W,\partial W,f)$ where $W$ is a manifold with boundary $\partial W$ and $f$ is an orientation preserving homotopy equivalence\footnote{Since we assume $X$ to be simply connected, every homotopy equivalence is simple and we do not need to require this in the definition.} that restricts to a diffeomorphism on the boundary. Two such tuples $(W_0,\partial W_0,f_0)$ and $(W_1,\partial W_1,f_1)$ are equivalent, if there exists a diffeomorphism $\alpha\colon W_0\to W_1$ such that $f_0=f_1\circ \alpha$ on the boundary $\partial W_0$ and on the whole $W_0$, the map $f_1\circ\alpha$ is homotopic to $f_0$ relative to $\partial W_0$ (\cite[Definition 10.2]{surgerybook}). It is a consequence of the $h$-cobordism theorem that for $\dim(M)\ge5$ we have the following isomorphism (\cite[Section 3.2, pp.33]{BerglundMadsen})
\[\pi_k\left(\frac{\haut(M)}{\blockdiff(M)}\right)\cong \calS_\partial(D^k\times M).\]
\noindent The main result of surgery theory is that the structure set $\calS_\partial(D^k\times M)$ fits into an exact sequence of sets known as the \emph{surgery exact sequence} (cf. \cite[Theorem 10.21 and Remark 10.22]{surgerybook}):
\begin{equation}\label{eq:ses}
	L_{k+d+1}(\bbZ)\too\calS_\partial(D^k\times M) \too \calN_\partial(D^k\times M)\overset{\sigma}{\too} L_{k+d}(\bbZ)
\end{equation}

\noindent Here, $\calN_\partial(D^k\times M)$ is the set of \emph{normal invariants} which is given by equivalence classes of tuples $(W, f,\hat f, \xi)$, where $W$ is a $(d+k)$-dimensional manifold with (stable) normal bundle $\nu_W$, $\xi$ is a stable vector bundle over $D^k\times M$ and $f\colon W\to D^k\times M$ is a map of degree $1$ that restricts to a diffeomorphism of the boundary and which is covered by a stable bundle map $\hat f\colon \nu_W\to \nu_{D^k\times M}\oplus\xi$. The equivalence relation is given by bordism of manifolds with cylindrical ends of degree $1$ normal maps(see \cite[Definition 10.7]{surgerybook}). 

Since we only consider simply connected manifolds, the relevant $L$-groups are $4$-periodic and given by (cf. \cite[Theorem 7.96]{surgerybook})
\[L_n(\bbZ) \cong\begin{cases}
	\bbZ &\quad\text{ if } n\equiv0\;(4)\\
	\bbZ/2 &\quad\text{ if } n\equiv2\;(4)\\
	0 &\quad\text{ otherwise}
\end{cases}\]

\noindent In particular, the map $\calS_\partial(D^k\times M) \embeds \calN_\partial(D^k\times M)$ if $(k+d)$ is even. The map $\sigma$ in the surgery exact sequence (\ref{eq:ses}) is the so-called \emph{surgery obstruction map}, which in degrees $d+k\equiv 0\;(4)$ with $k\ge1$ and for simply connected $M$ is given by
\[\sigma(W,f,\hat f,\xi) = \frac18\Bigl(\sign(\underbrace{W\cup (D^k\times M)}_{\eqqcolon W'}) - \sign(S^k\times M)\Bigr) =\frac18\sign(W')\]
where $\sign$ denotes the signature (cf. \cite[Lemma 7.170, Exercise 7.188]{surgerybook}). The signature of $W'$ can be computed via Hirzebruch's signature theorem, which constructs a power series
\begin{align*}
	\calL(x_1,x_2,\dots) ={}&1 + s_1x_1 + \dots + s_ix_i + \dots + s_{i,j}x_i\cdot x_j + \dots\\
		& + s_{i_1,\dots,i_n}x_{i_1}\cdots x_{i_n}+ \dots
\end{align*}
such that $\sign(W')=\scpr{\calL(p_1(TW'),p_2(TW'),\dots),\ [W']}$. Here $p_i(TW')$ are the Pontryagin classes of $W'$. %Note that $f, \hat f$, and $\xi$ can be extended trivially to $W'$. 

In order to further analyze $\calN_\partial(D^k\times M)$, let us define $G(n)=\{f\colon S^{n-1}\to S^{n-1}\text{ homotopy equivalence}\}$ and $\bg\coloneqq \colim_{n\to\infty} \bg(n)$. Note, that the index shift stems from the fact that one wants to have an inclusion $O(n)\subset G(n)$ of the orthogonal group. Analogously, let $\bo\coloneqq \colim_{n\to\infty}\bo(n)$. Note that $\bg$ is the classifying space for stable spherical fibrations whereas $\bo$ is the classifying space for stable vector bundles. The inclusion induces $J(n)\colon \bo(n)\embeds \bg(n)$ which in the colimit yields a map $J\colon \bo\to \bg$ and we denote its homotopy fiber by $\go$. By \cite[Equation 10.11]{surgerybook} there is an identification 
\[\calN_\partial(D^k\times M) \cong [(D^k,S^{k-1})\times M, (\go,\ast)],\]
where $[\_,\_]$ denotes homotopy classes of maps of pairs. For our purpose, we need a more explicit description of this identification, in particular we want to pay attention to the vector bundle data. We follow \cite[Theorem 6.17 and above]{surgerybook} for this. A map into $\go$ consists of a map $\gamma$ into $\bo$ and a homotopy $h$ in $\bg$ from $J\circ \gamma$ to the constant map. Since $D^k\times M$ is compact, $\gamma$ and $h$ actually land in finite stages $\bo(\ell)$ and $\bg(\ell)$. We obtain a vector bundle $\gamma^* U_\ell\to (D^k\times M)/(S^{k-1}\times M)$, where $U_\ell\to \bo(\ell)$ is the universal vector bundle and a trivialization $\overline h\colon S(\gamma^* U_d)\to \underline{S^{\ell-1}}$ as spherical fibrations. Next, we choose an embedding $D^k\times M\embeds \bbR^N_+\coloneqq \{(x_1,\dots, x_N)\colon x_1\ge0\}$ such that $S^{k-1}\times M$ embeds into $\{x_1=0\}$. The relative Pontryagin Thom construction yields a map $(D^N,S^{N-1})\to( \thom(S(\nu_{M\times D^k})),\thom(S(\nu_{M\times S^{k-1}})))$ into the Thom spaces of the respective sphere bundles. Using the map $\overline h$ from above, we can define a map 
\begin{align*}
	(D^{N+\ell},S^{N+\ell-1}) &= (D^\ell,S^{\ell-1})\times (D^N,  S^{N-1}) \too (D^\ell,S^{\ell-1})\times \thom(S(\nu_{D^k\times M}))\\
	\too\quad{}\ &(\thom(S(\nu_{ D^k\times  M})\ast \underline{S^{\ell-1}}),\thom(S(\nu_{ S^{k-1}\times M})\ast \underline{S^{\ell-1}})) \\
	\overset{\thom(\id\ast\overline h^{-1})}\too {}&(\thom(S(\nu_{D^k\times M} \oplus \gamma^*U_\ell)),\thom(S(\nu_{S^{k-1}\times M} \oplus \gamma^*U_\ell))),
\end{align*}
where the  middle map  is induced by the projection $D^\ell\times (D(\nu_{D^k\times M})/S(\nu_{D^k\times M})) \to D(D^\ell\times\nu_{D^k\times M})/S(D^\ell\times\nu_{D^k\times M})$. The reverse of the Pontryagin Thom-construction yields an element $(W, \hat f, f, \gamma^*U_\ell)\in\calN_\partial(D^k\times M)$. Note that the bundle $\xi$ is precisely given by the (stable) vector bundle classified by $\gamma\colon D^k\times M\to \bo(\ell)\to \bo$.

Next, we identify $S^k\wedge M_+ = (S^k\times M_+)/(\{1\}\times M_+\cup S^k\times\{+\}) \cong (D^k\times M)/ (S^{k-1}\times M)$, where $M_+$ is $M$ with a disjoint base point and $\wedge$ denotes the smash product of pointed spaces. The functor $S^k\wedge{(\_)}$ is adjoint to the $k$-fold loop space functor $\Omega^k(\_)$ and so we get $[S^k\wedge M_+, \go]_* \cong [M,\Omega^k \go]$. Now $\Omega^{k+1} \bg$ is the homotopy fiber of the map $\Omega^k\go\to \Omega^k\bo$ and by obstruction theory (cf. \cite[p. 418]{hatcher_at}) the obstructions to the lifting problem
\begin{center}
\begin{tikzpicture}
	\node (0) at (0,0) {$M$};
	\node (1) at (3,0) {$\Omega^k \bo$};
	\node (2) at (3,1.2) {$\Omega^k \go$};
	
	\draw[->] (0) to (1);
	\draw[->] (2) to (1);
	\draw[->, dashed] (0) to (2);	
\end{tikzpicture}
\end{center}
live in the groups $H^{i+1}(M;\pi_i(\Omega^{k+1}\bg))\cong H^{i+1}(M;\pi_{k+i+1}(\bg))$. The homotopy groups $\pi_{k+i+1}(\bg)$ are isomorphic to the shifted stable homotopy groups of spheres $\pi_{k+i}^{st}$ by \cite[p. 135]{surgerybook}. By Serre's finiteness theorem, these groups are finite for $k+i\ge1$ and hence all of our obstruction groups vanish rationally, since we assumed that $k\ge1$. Since $\maps_*(M,\Omega^k \bo)$ is an $H$-space, we see that for every (pointed) map $f\colon M\to \Omega^k\bo$, some multiple of $f$ can be lifted to $\Omega^k \go$. Therefore it suffices for us to specify an element in
\[[(D^k, S^{k-1})\times M, (\bo,\ast)] = {\ko}^0((D^k, S^{k-1})\times M)\]
in order for a multiple of this element to yield a normal invariant $(W,f,\hat f,\xi)$. Furthermore, by the discussion above, the bundle $\xi$ in this normal invariant is precisely given by the multiple of the element in ${\ko}^0((D^k, S^{k-1})\times M)$ and $\xi$ is trivial when restricted to $S^{k-1}\times M$. For such a normal invariant we can hence extend $\xi$ by the trivial bundle to a bundle $\xi'$ over $W'\coloneqq W\cup_f D^k\times M$ and the maps $f$ and $\hat f$ can  be extended by the identity to  a (stable) degree one normal map $(\hat f',f')\colon \nu_{W'}\to\nu_{S^k\times M}\oplus \xi'$. 

Next, we consider the isomorphism given by the Pontryagin character:
\begin{align*}
	\ph(\_)\coloneqq \ch(\_\otimes\bbC)\colon {\ko}^0((D^k,S^{k-1})\times M)\otimes\bbQ\congarrow \bigoplus_{i\ge0} &H^{4i}((D^k,S^{k-1})\times M;\bbQ)\\
		&\cong u_k\times \bigoplus_{i\ge0} H^{4i-k}(M;\bbQ)
\end{align*}
for $u_k$ the cohomological fundamental class in $H^k(D^k,S^{k-1})\cong H^k({S^k,\ast})$. The $i$-th component of the Pontryagin character is given by
\begin{align*}
	\ph_i(\xi) &= \ch_{2i}(\xi\otimes\bbC) = \frac{1}{(2i)!}\Bigl(\bigl(-2i)c_{2i}(\xi) + f(c_1(\xi),\dots, c_{2i-1}(\xi)\bigr)\Bigr)\\
	 &= \frac{(-1)^{i+1}}{(2i-1)!} p_i(\xi) 
\end{align*}
where $f(c_1(\xi),\dots, c_{2i-1}(\xi))$ is a polynomial in Chern classes of $\xi$ homogenous of degree $2i$ which vanishes since all nontrivial products in $H^*((D^k, S^{k-1})\times M;\bbQ)$ are trivial. Hence, for any collection $(x_i)\in H^{4i-k}(M;\bbQ)$ and $(A_i)\in \bbQ$ there exists a $\lambda\in\bbZ\setminus\{0\}$ and a normal invariant $ (W,f,\hat f,\xi)\in\calN_\partial(D^k\times M)$ such that the bundle $\xi$ has the following Pontryagin classes
\[p_i(\xi) = (-1)^{i+1}(2i-1)!\lambda A_i\cdot u_k\times x_i,\]
\noindent where $u_k$ denotes the cohomological fundamental class of $S^k$. We observe, that $p_i(\xi)\cup p_j(\xi)=0$ for $i,j\ge1$, since $u_k^2=0$. Furthermore, $(-1)^{i+1}(2i-1)!\not=0$ for all choices of $i$ and hence, after replacing $A_i$ by $\frac{(-1)^{i+1} A_i}{(2i-1)!}$ we may assume that the Pontryagin classes of $\xi$ have the form\footnote{Note that $\lambda$ depends on the collection $(A_i)$ so we cannot absorb it into the $A_i$'s}
\begin{align}\label{eq:p-classes-xi}
	p_i(\xi) = \lambda A_i\cdot u_k\times x_i
\end{align}
This allows us to construct a normal invariant in the kernel of the signature homomorphism and hence an element of $\pi_k(\haut(M)/\blockdiff(M))$ such that the underlying stable vector bundle has prescribed Pontryagin classes, which we will do in the succeeding section.

With  regard to Pontryagin numbers of the extension $W'$ of $W$, we remark that $p_i(TW') = p_i(-\nu_{W'})$ and $p_i(-(\nu_{S^k\times M}\oplus \xi'))= p_i(T(S^k\times M)\oplus -\xi) = p_i(\pr^*TM\oplus -\xi')$ for $\pr\colon S^k\times M\to M$. Since $\hat f'\colon \nu_{W'}\to\nu_{S^k\times M}\oplus \xi'$ is of degree one, any Pontryagin number of $W'$ equals the corresponding Pontryagin number of $\pr^*TM\oplus-\xi'$. Furthermore, as $\xi'$ is trivial on the complement of $W\subset W'$, the Pontryagin numbers of $\xi'$ are obtained from the ones of $\xi$ by replacing the fundamental class in $H^k(D^k,S^{d-1})$ by the corresponding one in $H^k(S^k,\ast)$ in \pref{Equation}{eq:p-classes-xi}.

%As a final observation in this section, we note that the surgery obstruction map $\sigma\colon \calN_\partial(D^k\times M)\to \bbZ$ is always surjective by Wall's realization theorem (cf. \cite[Theorem 7.192]{surgerybook}). Therefore we have
%\[\dim\ker(\sigma\colon\calN_\partial(D^k\times M)\to\bbZ)\otimes\bbQ=\dim \bigoplus_{i\ge0} H^{4i-k}(M;\bbQ) - 1\]
%and hence $\dim\pi_k\left(\haut(M)/\blockdiff(M)\right)\otimes\bbQ = \sum_{i\ge0} b_{4i-k}(M;\bbQ) -1$. Together with \pref{Lemma}{lem:unblocking} this implies 
%\begin{cor}
%	For $k\le\min(\frac{d+2}{3},\frac{d-3}{2})$ we have 
%	\[\dim\pi_k(B\!\diff(M))\otimes\bbQ \ge \sum_{i\ge0} b_{4i-k}(M;\bbQ) -1.\]
%\end{cor}
%\begin{proof}
%	Consider the  
%	\[\pi_k\left(\haut(M)/\blockdiff(M)\right)\otimes\bbQ\]
%
%	Given a class $(x_i)\in \bigoplus_{i\ge0}H^{4i-k}(M;\bbQ)$ such that the corresponding normal invariant has vanishing signature, we obtain a (fiber homotopy trivial) bundle $E\to S^k$, classified by an element in $\pi_k(\haut(M)/\diff(M))\otimes\bbQ$ which yields an element in $\pi_k(\bdiff(M))\otimes\bbQ$. If two such bundles $E$ and $E'$ are isomorphic as fiber bundles, this would enforce $p_i(TE)=p_i(TE')$ for all $i$. Hence $x_i$ and $x_i'$ were multiples for all $i$ and in particular linearly dependent.\todo{Is this proof valid, need to talk to Wolfgang.}
%\end{proof}

%\section{Proof of main theorem}\label{sec:main}
\section{Prescribing Pontryagin classes}\label{sec:prescribing}

\noindent In this section we will prove the block-analogues of our main results. Recall that a characteristic class $c\in H^{d+k}(\bo(d);\bbQ)$ is called \emph{block-spherical (resp. block-$h$-spherical) for $M$} if there exists an $M$-block-bundle $E \to S^k$ (resp. a blockwise homotopy trivial one) such that $\scpr{c(E),[E]}\not=0$.

\subsection{Proof of \pref{Theorem}{thm:main}}

\pref{Theorem}{thm:main}, (i) follows from the following \enquote{block-version} in combination with \pref{Lemma}{lem:unblocking} (see also \pref{Remark}{rem:unblocking}).

\begin{lem}\label{lem:realize-p}
	Let $4m=d+k$ and let $p_m\not=p=p_{i_1}\cdot\hdots\cdot p_{i_s}\in H^{4m}(BSO;\bbQ)$ be a monomial in universal Pontryagin classes. Then 
	\[p\text{ is block-$h$-spherical} \quad\iff\quad \begin{matrix}\text{There exists an }\ell\ge1\text{ such that }\colon i_\ell\ge\frac k4\text{ and } \\p_{i_1}(TM)\cup\hdots\cup\widehat{p_{i_\ell}(TM)}\cup\hdots\cup p_{i_s}(TM) \not=0\end{matrix}\]
\end{lem}

\begin{proof}
	We first show the \enquote{$\Leftarrow$} implication. Let $\ell$ be such that $i_\ell$ is maximal with the property above and let $x\in H^*(M;\bbQ)$ be such that 
	\[p_{i_1}(TM)\cup\hdots\cup\widehat{p_{i_\ell}(TM)}\cup\hdots\cup p_{i_s}(TM)\cup x = u_M\in H^d(M;\bbQ).\]
	By the discussion in \pref{Section}{sec:surgery} (in particular, see \pref{Equation}{eq:p-classes-xi}) there exists a normal invariant $\eta$ such that the underlying (extended) stable vector bundle $\xi\to S^k\times M$ has the following Pontryagin classes:
	\begin{align*}
		p_0(-\xi) &= 1 \qquad	p_{i_\ell}(-\xi) = \lambda \cdot u_k\times x\\
		p_m(-\xi) &= \lambda A\cdot u_k\times u_M
	\end{align*}
	for $u_k$ the cohomological fundamental class of $S^k$, $A\in\bbQ$ to be chosen later and $\lambda\in\bbZ\setminus\{0\}$ determined by $A$. All other Pontryagin classes of $\xi$ vanish. Note that by assumption $i_\ell<m$. Then\footnote{Since we are only interested in rational Pontryagin classes the Whitney sum formula $p(V\oplus W)=p(V)\cup p(W)$ holds.}	
	\reqnomode
	\begin{align}\label{eq:1}
		p_{i_1}\cup\hdots\cup &p_{i_s}(\pr^*TM\oplus -\xi) = \prod_{j=1}^s\sum_{n=0}^{i_j} p_n(-\xi)\cup p_{i_j-n}(\pr^*TM)
	\end{align}
	for $\pr\colon S^k\times M\to M$ the projection. Since $p_i(\pr^*TM)=\pr^*p_i(TM)=1\times p_i(TM)$, we will from now on omit \enquote{$\pr^*$} in our computations. Since $i_j<m$ and by our choice of Pontryagin classes of $\xi$, the expression $p_n(-\xi)\cup p_{i_j-n}(TM)$ is nonzero only if $n=0$ or if $n=i_\ell$. Furthermore, recall that $p_n(-\xi)\cup p_{n'}(-\xi)=0$ for $n,n'\ge1$. We continue the computation
	\begin{align}\label{eq:2}
		(\ref{eq:1})&= \prod_{i_j\ge i_\ell} \left(p_{i_j}(TM) + p_{i_\ell}(-\xi)\cup p_{i_j-i_\ell}(TM)\right)\cup\prod_{i_j<i_\ell}p_{i_j}(TM)\notag\\
			&=\prod_{j=1}^sp_{i_j}(TM) + \prod_{i_j<i_\ell}p_{i_j}(TM) \cup\sum_{i_j\ge i_\ell}\left(p_{i_\ell}(-\xi)\cup p_{i_j-i_\ell}(TM)\cup \prod_{q\not=j}p_{i_q}(TM)\right)
	\end{align}
	where the second equality holds after multiplying out and using the fact that there can only be one factor $p_{i_\ell}(-\xi)$ in each summand. The first summand vanishes for degree reasons. If $i_j>i_\ell$, then $\prod_{q\not=j}p_{i_q}(TM)=0$ because we chose $i_\ell$ to be maximal such that this product does not vanish. Hence, the latter factor vanishes if $i_j>i_\ell$ and we get:
	\begin{align*}
			(\ref{eq:2})&= p_{i_\ell}(-\xi)\cup \prod_{i_j<i_\ell}p_{i_j}(TM) \cup\sum_{j\colon i_j= i_\ell}\left(\prod_{q\not=j}p_{i_q}(TM)\right)\\
			&=\underbrace{p_{i_\ell}(-\xi)}_{=\lambda\cdot x\times u_k}\cup\prod_{q\not=\ell}p_{i_q}(TM) \cdot\underbrace{\sum_{j\colon i_j=i_\ell} 1}_{\eqqcolon a_\ell\not=0} = \lambda a_\ell \cdot u_k\times u_M \not=0
	\end{align*}
	Finally, we need to choose $A$, such that the surgery obstruction vanishes. We have 
	\[p_m(TM\oplus -\xi) = p_m(-\xi) + p_{m-i_\ell}(TM)\cup p_{i_\ell}(-\xi),\]
Consider Hirzebruch's signature formula:
	\begin{align*}
	\sigma(\eta) ={}& \sign(W') = \scpr{\calL(W'),\ [W']} = \scpr{\calL(S^k)\cdot\calL(TM\oplus-\xi),\ [S^k\times M]}\\
		={}& s_m\cdot \scpr{p_m(TM\oplus-\xi),\ [S^k\times M]}\\
		&\quad + \scpr{\calL(TM\oplus-\xi)-s_m\cdot p_m(TM\oplus\xi),\ [S^k\times M]}\\
		={}&s_m\cdot\scpr{p_m(-\xi),[S^k\times M]}\\
		&+\underbrace{\scpr{\calL(TM\oplus-\xi)-s_m\cdot p_m(TM\oplus\xi) + s_m\cdot p_{m-i_\ell}(TM)\cup p_{i_\ell}(-\xi)}}_{\eqqcolon z}\\
		={}& s_m\lambda\cdot A + z, 
	\end{align*}
	where $s_m\not=0$ is the leading coefficient of $\calL$. Note that $z=\lambda\cdot z_0$ for some $z_0$ which is independent of $A$. This is true, since there appears precisely one factor $p_{i_\ell}(\xi)$ (which is a multiple of $\lambda$) in every monomial summand of $(\calL(TM\oplus\xi)-s_mp_m(TM\oplus \xi)+s_m\cdot p_{m-i_\ell}(TM)\cup p_{i_\ell}(-\xi))$ while all other factors are coefficients of $\calL$ or Pontryagin classes of $M$ which are independent of $\lambda$. We choose $A\coloneqq \frac{z_0}{s_m}$ so that $\sigma(\eta)$ vanishes independently of $\lambda$. We hence obtain a normal invariant with vanishing signature and therefore a block bundle with the desired properties.
	
	For the other implication \enquote{$\Rightarrow$} let us assume that $p_{i_1}(TM)\cup\hdots\cup\widehat{p_{i_\ell}(TM)}\cup\hdots\cup p_{i_s}(TM)=0$ for all $i_\ell\ge \frac k4$. 
	\[p_{i_1}\cup\hdots\cup p_{i_s}(TM\oplus -\xi) = \prod_{j=1}^s\sum_{n=0}^{i_j} p_n(-\xi)\cup p_{i_j-n}(TM)\]
	Since $p_n(\xi)\cup p_{n'}(\xi)=0$ for $n,n'\ge1$ and $p_{i_1}(TM)\cdots p_{i_s}(TM)=0$ (for degree reasons), every summand in the above expression must contain precisely on factor $p_n(\xi)$ for some $n\ge1$. Hence, multiplying out the above delivers
	\begin{align*}
		p_{i_1}\cup\hdots\cup &p_{i_s}(TM\oplus -\xi) = \sum_{j=1}^s\sum_{n=1}^{i_j}p_n(-\xi)\cup p_{i_j-n}(TM)\cup  \prod_{r\not= j}p_{i_r}(TM).
	\end{align*}
	The product $\prod_{r\not= j}p_{i_r}(TM)$ vanishes if $i_j\ge\frac k4$ by assumption. If $i_j<\frac k4$, then $p_n(-\xi)=0$ for all $1\le n\le i_j$ since every higher Pontryagin class of $\xi$ is of the form $u_k\times *$ and hence is of index at least $k/4$. It follows that there is no normal invariant with $p_{i_1}\cup \cdots \cup p_{i_s}(TW)\not=0$ and hence there can be no such block bundle.% since the homomorphism $\calS_\partial(D^k\times M) \to\calN_\partial(D^k\times M)$ is injective by our assumption on $d+k$.
\end{proof}

\noindent Next we turn to the proof of \pref{Theorem}{thm:main}(ii) and \pref{Theorem}{main:lowerbound} which will again follow from block-analogues thereof combined with \pref{Lemma}{lem:unblocking}. Recall the following definitions 
	\begin{align}\label{align:imin}
	\begin{split}
		i_{\min}&\coloneqq \min\{i\ge1 \colon p_i(TM)\not=0\}\\
		n_{\max}&\coloneqq \max\{n\in\bbN \colon p_{i_{\min}}(TM)^n\not=0\}.
	\end{split}
	\end{align}
	We assume that $M$ admits at least on nontrivial rational Pontryagin class, so the set $\{i\ge1 \colon p_i(TM)\not=0\}$ is actually non-empty and $n_{\max}\ge1$.

\begin{lem}\label{lem:pmin}
	For every $\ell=1,\dots, n_{\max}$ there exists a normal invariant $\eta_\ell$ with underlying (extended) stable vector bundle $\xi_\ell\to S^k\times M$ with the following property:
	\begin{align*}
		\scpr{p_{i_{min}}^r(TM\oplus -\xi_\ell)\cup &p_{m-r\cdot i_{\min}}(TM\oplus -\xi_\ell),\ [S^k\times M]}\not=0\quad\iff\quad \substack{r=\ell \\\text{ or } r=0}
	\end{align*}
	and $\sigma(\eta_\ell)=0$. For $\ell=1$ we furthermore have that $\xi_1$ can be chosen such that 
	\[\scpr{p_{i_{\min}}(TM\oplus -\xi_1)\cup p_{m-i_{\min}}(TM\oplus -\xi_1),\ [S^k\times M]}\text{ and } \scpr{p_{m}(TM\oplus -\xi_1),\ [S^k\times M]}\]
	are the \emph{only} non-vanishing monomial Pontryagin numbers of $TM\oplus -\xi_1$.
\end{lem}

\begin{proof}
	Let $u_M\in H^{4m-k}(M;\bbQ)$ denote the cohomological fundamental class of $M$. Since the cup product induces a perfect pairing 
	\[H^{4j}(M;\bbQ)\times H^{4(m-j)-k}(M;\bbQ)\to\bbQ,\]
	there exists a class $x\coloneqq x_{n_{\max}}\in H^{4(m-i_{\min}\cdot n_{\max})-k}(M;\bbQ)$ such that $x\cup p_{i_{\min}}(TM)^{n_{\max}}=u_M$. For $r = 0,\dots, n_{\max}$ we define $x_r\coloneqq x \cup p_{i_{\min}}(TM)^{n_{\max}-r}$. Then 
	\[x_r\cup p_{i_{\min}}(TM)^r = x \cup p_{i_{\min}}(TM)^{n_{\max}-r}\cup p_{i_{\min}}(TM)^r = u_M.\]
	By the discussion in \pref{Section}{sec:prelim} we know that for every collection $A_0,\dots, A_{n_{\max}}\in\bbQ$ there exists a $\lambda\in\bbZ\setminus\{0\}$ and a normal invariant $\eta_\ell=(W_\ell,f_\ell,\hat f_\ell,\xi_\ell)$ such that the (extended) stable vector bundle $\xi_\ell'$ has only the following (rational) Pontryagin classes:
	\begin{align*}
		p_0(-\xi_\ell') &= 1\\
		p_{m-r\cdot i_{\min}}(-\xi_\ell') &= \lambda A_r \cdot u_k\times x_r \text{ for } r=0,\dots, n_{\max}	\end{align*} 
	\noindent Since $r\cdot i_{\min}<m$ and $p_q(TM)=0$ for all $0<q<i_{\min}$, we have
	\begin{align*}
		p_{i_{min}}(TM\oplus-\xi_\ell') &= \sum_{a=0}^{i_{\min}} p_a(TM)\cup p_{i_{\min}-a}(-\xi_\ell')
%			\\&= p_{i_{\min}} (\xi_\ell') + \sum_{a=i_{\min}}^{i_{\min}} p_a(TM)\cup p_{i_{\min}-a}(-\xi_\ell')\\&
			= p_{i_{\min}} (-\xi_\ell') + p_{i_{\min}} (TM) 
	\end{align*}
	We will now distinguish two cases: $m=(s+1)\cdot i_{\min}$ for some $1\le s\le n_\max$ and $m\not=(r+1)\cdot i_{\min}$ for all $r$. In the former case we have $p_{i_{\min}}(-\xi_\ell')=\lambda A_{s}\cdot u_k\times x_{s}$ and compute:
	\begin{align*}
		p_{i_{\min}}(TM\oplus &-\xi_\ell')^{s}\cup p_{\underbrace{m-s\cdot i_{\min}}_{=i_{\min}}}(TM\oplus -\xi_\ell') = p_{i_{\min}}(TM\oplus -\xi_\ell')^{s +1}\\
			&= (p_{i_{\min}} (-\xi_\ell') + p_{i_{\min}} (TM) )^{s+1} =(s+1)\cdot p_{i_{\min}}(TM)^{n_{\max}}\cup p_{i_{\min}}(-\xi_\ell')\\
			&=(s+1)\lambda A_{s}\cdot \underbrace{p_{i_{\min}}(TM)^{s}\cup u_k\times x_{s}}_{=u_k\cdot u_M} 
	\end{align*}
	where the third equality follows from multiplying out and the fact that $p_n(\xi)\cup p_{n'}(\xi)=0$ for $n,n'\ge1$. For $0\le r\not=s$ we have:
	\begin{align}\label{eq:computation1}
		p_{i_{\min}}(TM&\oplus -\xi_\ell')^r \cup p_{m-i_{\min}\cdot r}(TM\oplus -\xi_\ell)\notag\\
			={}& \bigl(p_{i_{\min}} (-\xi_\ell') + p_{i_{\min}} (TM) \bigr)^r\cup \sum_{a=0}^{m-i_{\min}\cdot r}p_a(TM) \cup p_{{m-i_{\min}\cdot r}-a}(-\xi_\ell')\notag\\
			={}& \bigl(p_{i_{\min}} (-\xi_\ell') + p_{i_{\min}} (TM) \bigr)^r\\
				&\cup\Bigl(p_{{m-i_{\min}\cdot r}}(-\xi_\ell') + \sum_{a=i_{\min}}^{m-i_{\min}\cdot r}p_a(TM) \cup p_{{m-i_{\min}\cdot r}-a}(-\xi_\ell')\Bigr)\notag\\
			={}& \underbrace{ p_{i_{\min}} (TM)^r\cup \lambda A_r u_k\times x_r}_{ = \lambda A_r \cdot u_k\times u_M}\ +\ \ast\cdot u_k\times u_M,\notag
	\end{align}
	where $\ast$ is a linear expression in the variables $\lambda A_{r+1},\dots,\lambda A_{n_{\max}}$. Now for $b,a_1,\dots, a_{n_{\max}}\in\bbQ$, consider the following system of equations
	\begin{align*}
		b	&=	\scpr{p_{m}(TM\oplus-\xi_\ell'), [S^k\times M]}\\
		a_1 & = \scpr{p_{i_{\min}}(TM\oplus-\xi_\ell')\cup p_{m-i_{\min}}(TM\oplus-\xi_\ell'), [S^k\times M]}\\
			&\qquad\vdots\\
		a_r & = \scpr{p_{i_{\min}}(TM\oplus-\xi_\ell')^r\cup p_{m-r\cdot i_{\min}}(TM\oplus-\xi_\ell'), [S^k\times M]}\\
			&\qquad\vdots\\
		a_{n_{\max}} &= \scpr{p_{i_{\min}}(TM\oplus-\xi_\ell')^{n_{\max}}\cup p_{i_{\min}}(TM\oplus-\xi_\ell'), [S^k\times M]} 
	\end{align*}
	By the above computation this is a linear system of equations in the variables $\lambda A_0,\lambda A_1,\dots, \lambda A_{n_{\max}}$ and it has the following form
	\begin{align*}
		\underbrace{\begin{pmatrix}
			1 & \ast & \ast & \ast & \ast \\
			0 & \ddots & \ast & \ast & \ast  \\
			0 & 0 & s+1 &  \ast & \ast \\
			0 & 0 & 0 & \ddots &\ast  \\
			0 & 0 & 0 & 0 & 1 \\
		\end{pmatrix}}_{\eqqcolon B}
		\cdot
		\begin{pmatrix}
			\lambda A_0\\\vdots\\\lambda A_{n_{\max}}
		\end{pmatrix}
		=
		\begin{pmatrix}
			b\\a_1\\\vdots\\a_{n_{\max}}
		\end{pmatrix}
	\end{align*}
	with $s+1$ in the $s$-th row. The matrix $B$ is invertible and hence, we can choose $A_1,\dots A_{n_{\max}}$ such that $a_i=0$ if and only if $i\not=\ell$. Note, that $\lambda$ is not yet determined as it also depends on $A_0$, but the condition of $a_i$ being zero or nonzero is independent of $\lambda$. Furthermore note, that since $B$ is triangular, the values of $a_i$ are independent of $A_0$. Finally, we need to choose $A_0$, such that the surgery obstruction vanishes. Consider Hirzebruch's signature formula:
	\begin{align*}
	\sigma(\eta_\ell) &= \sign(W_\ell') = \scpr{\calL(W_\ell'),\ [W_\ell']} = \scpr{\calL(W_\ell'),\ f_*[S^k\times M]}\\
		&= \scpr{\calL(S^k)\calL(TM\oplus\xi_\ell),\ [S^k\times M]} = \scpr{\calL(TM\oplus\xi_\ell),\ [S^k\times M]} \\
		&= s_m\cdot \lambda\cdot A_0 + \lambda\cdot z_0 
	\end{align*}
	where $z_0$ is some number independent of $A_0$ and $s_m$ is the leading coefficient of $\calL$ as in the proof of \pref{Lemma}{lem:realize-p}. Since $s_m\not=0$, we can choose $A_0\coloneqq \frac{z_0}{s_m}$ so that $\sigma(\eta_\ell)$ vanishes independently of $\lambda$. 
	
	The case $m\not=(r+1)\cdot i_{\min}$ for all $r$ is very similar. By the same computation as (\ref{eq:computation1}) we have for all $r\ge0$:
	\begin{align*}
p_{i_{\min}}(TM&\oplus -\xi_\ell')^r \cup p_{m-i_{\min}\cdot r}(TM\oplus -\xi_\ell)\\
	={}& \underbrace{p_{i_{\min}} (TM)^r\cup \lambda A_r \cup u_k\times x_r}_{ = \lambda A_r \cdot u_k\times u_M}\ +\ \ast\cdot u_k\times u_M,
	\end{align*}
	This implies that the respective matrix $B$ has the same form as above with $s+1$ replaced by $1$. The rest of the argument is verbatim to the first case.
	
	If $\ell=1$ the above system of linear equations reduces to $a_2=a_3=\dots=a_{n_\max}=0$. From the shape of $B$, this implies that $A_2=\dots=A_{n_\max}$ must be $0$ as well and hence bundle $\xi_1'$ only has three non-vanishing Pontryagin classes, namely
	\begin{align*}
		p_0(\xi_1') &=1 \\
		p_{m-i_{\min}}(-\xi_1') &= \lambda A_1 \cdot u_k\times x\\
		p_{m}(-\xi_1') &= \lambda A_0 \cdot u_k\times u_M
	\end{align*}
	As noted above, every Pontryagin number of $TM\oplus-\xi_1'$ contains precisely one Pontryagin class of $\xi_1'$. Since $p_{i_{min}}(TM)$ is the smallest Pontryagin class of $M$, we deduce that the only possibly non-vanishing Pontryagin numbers of $TM\oplus-\xi_1'$ are $\scpr{p_{i_{\min}}(TM\oplus -\xi_1')\cup p_{m-i_{\min}}(TM\oplus -\xi_1'),\ [S^k\times M]}$ and $ \scpr{p_{m}(TM\oplus -\xi_1'),\ [S^k\times M]}$. By construction, 
	\[\scpr{p_{i_{\min}}(TM\oplus -\xi_1')\cup p_{m-i_{\min}}(TM\oplus -\xi_1'),\ [S^k\times M]}\not=0\]
	Since $\sigma(\eta_1)=0$ by construction, we deduce that $\scpr{p_{m}(TM\oplus -\xi_1'),\ [S^k\times M]}$ is nonzero as well, since the leading coefficient in the $\calL$-polynomial is nontrivial by \cite[p. 12]{hirzebruch_topological}.
\end{proof}

\subsection{Proof of \pref{Corollary}{cor:lowerbound}, (ii)}\label{sec:proofnmax}
We need to show $\dim\fib_{M,k}\le n_\max$. Let $p_I[E]\coloneqq\scpr{p_I(TE),[E]}$ be a monomial Pontryagin number of a fiber bundle $M\overset{\iota}\to E^{4m}\overset{\pi}\to S^k$. 
%Let $h\colon E\overset{\sim}\to M\times S^k$ be a fiber homotopy trivialisation. Therefore we have isomorphisms of rational cohomology rings
%\begin{align*}
%	H^{*}(E) &\cong H^*(M\times S^k) \cong \bigoplus_{p+q=*}H^p(S^k)\otimes_\bbQ H^q(M)\\
%		&\cong H^*(M) \oplus \left(H^k(S^k)\otimes_\bbQ H^{*-k}(M)\right)
%\end{align*}
%where the second isomorphism is given by the $\times$-product. The projection onto the first factor of this sum yields the map $\iota^*\colon H^*(E)\to H^*(M)$. Therefore, any class 
We consider the Wang-sequence:
\[\dots\too H^n(M)\overset{\delta}\too H^{k+n}(E)\overset{\iota^*}\too H^{k+n}(M)\too \dots.\]  Furthermore,
\[\iota^*p_i(TE) = \iota^*p_i(\pi^*TS^k\oplus T_{\pi}E) = \iota^*p_i(T_{\pi}E) = p_i(\iota^*T_{\pi}E) = p_i(TM)\]
for $T_{\pi}E$ the vertical tangent bundle of $E$. The fiber homotopy trivialization $h\colon E\to M\times S^k$ yields $s\coloneqq \pr_M\circ h\colon E\to M$ and $(s\circ \iota)^*TM=TM$. Therefore $p_i(TM)=\iota^*p_i(s^*TM)$ and by the above computation it follows for all $i$ that $(p_i(TE)-p_i(s^*TM))=\delta(x_i)$ for some $x_i\in H^{4i-k}(M)$. Since products of elements in the image of $\delta$ vanish by \cite[p. 337, Corollary 3.3]{whitehead_elements}, we have
\[p_I(TE) = \prod_{i\in I}(\delta(x_i)+p_i(s^*TM)) = s^*p_I(TM) + \sum_{i\in I}\delta(x_i)\cup s^*p_{I\setminus \{i\}}(TM).\]
The first summand vanishes for degree reasons and $p_{I\setminus \{i\}}(s^*TM)=a_i\cdot p_{i_\min}(s^*TM)^{n_i} = a_i\cdot (p_{i_\min}(TE) - \delta(x_{i_\min}))^{n_i}$ for some $a_i\in\bbQ$ and $n_i=\frac{m-i}{i_\min}$ by our assumption on $M$. Therefore,
\begin{align*}
	p_I(TE)={}&\sum_{i\in I}a_i\cdot\delta(x_i)\cup p_{i_\min}(TE)^{n_i} \\
		={}& \sum_{i\in I}a_i\cdot(p_i(TE) - p_i(s^*TM))\cup p_{i_\min}(TE)^{n_i}\\
		={}& \sum_{i\in I}a_i\cdot p_i(TE)\cup p_{i_\min}(TE)^{n_i} - \sum_{i\in I}a_i\cdot p_i(s^*TM)\cup p_{i_\min}(TE)^{n_i}
\end{align*}
For the  second sum in this formula, we note  that there exist a $b_i\in \bbQ$ such that for $m_i=\frac{i}{i_\min}$ we have 
\begin{align*}
	\sum_{i\in I}a_i&\cdot p_i(s^*TM)\cup p_{i_\min}(TE)^{n_i} \\
	={}&\sum_{i\in I}a_i b_i p_{i_\min}(s^*TM)^{m_i}\cup (\delta(x_{i_\min}) + p_{i_\min}(s^*TM))^{n_i}\\
	={}&\sum_{i\in I}a_i b_i \binom{n_i}1\delta(x_{i_\min})\cup p_{i_\min}(s^*TM)^{m_i+n_i-1}.
\end{align*}
Note that $m_i+n_i=\frac{m}{i_\min}\eqqcolon \ell$ is independent of $i$ and we compute
\begin{align*}
	\sum_{i\in I}a_i&\cdot p_i(s^*TM)\cup p_{i_\min}(TE)^{n_i} \\
	={}&\delta(x_{i_\min})\cup p_{i_\min}(s^*TM)^{\ell-1}\sum_{i\in I}a_i b_i n_i\\
	={}&\left(\delta(x_{i_\min})+ p_{i_\min}(s^*TM)\right)^{\ell}\underbrace{\frac1\ell\sum_{i\in I}a_i b_i n_i }_{\eqqcolon \lambda} = \lambda\cdot p_{i_\min}(TE)^\ell.
\end{align*}
and in total
\[p_I(TE) = \sum_{i\in I}a_i\cdot p_i(TE)\cup p_{i_\min}(TE)^{n_i}-\lambda\cdot p_{i_\min}(TE)^\ell. \]
Since $a_i$ and $\lambda$ do not depend on $E$ but only on $M$ und $I$, the restriction of the functional $p_I[\_]$ to $\fib_{M,k}$ is contained in the linear span of the functionals $(p_{i_\min}^e\cup p_{m-e\cdot i_\min}[\_])_{e=0..n_\max}$ and hence $\dim(\fib_{M,k})^*\le n_\max+1$. The signature of any fiber bundle $E\to S^k$ is trivial and equals the $\calL$-polynomial in Pontryagin classes of $TE$ evaluated against $[E]$. Since every coefficient in the $\calL$-polynomial is nontrivial by \cite{BergBerg}, this gives one linear relation among the restricted functionals $(p_{i_\min}^e\cup p_{m-e\cdot i_\min}[\_])_{e=0\dots n_\max}$ and it follows that 
\[\dim\fib_{M,k} = \dim(\fib_{M,k})^*\le n_\max\]
which proves the claim.\hfill$\Box$

\subsection{Bounds on \texorpdfstring{$\fib_{M,k}$}{the dimension of M-bundles over the k-sphere}} Let $\blockfib{M,k}\subset\Omega_{d+k}\otimes\bbQ$ denote the linear subspace spanned by blockwise homotopy trivial $M$-block bundles. By \pref{Lemma}{lem:unblocking}, $\blockfib{M,k}\subset\fib_{M,k}$ and by \pref{Lemma}{lem:pmin}, $\dim(\blockfib{M,k})\ge n_\max$. In this section, we prove the upper bounds claimed in \pref{Theorem}{main:upperbound} or rather its block-analogue.

\begin{prop}\label{prop:upper-obstruction}
	Let $p = p_{i_1}\cup\hdots\cup p_{i_s}\in H^{4m}(BO(d);\bbQ)$. If $i_j<m-\frac{d}{4}$ for all $j\in\{1,\dots,s\}$, $p_{i_1}(TE)\cup\dots\cup p_{i_s}(TE)=0$ for all blockwise homotopy trivial $M$-block bundles $E$.
\end{prop}

\begin{proof}
	If $p_{i_1}(TE)\cup\dots\cup p_{i_s}(TE)$ were nonzero, we would get an $i_j$ such that $p_{i_1}(TM)\cup\hdots\cup\widehat{p_{i_j}(TM)}\cup\hdots\cup p_{i_s}(TM)\not=0$ by \pref{Lemma}{lem:realize-p}. However the degree of this product is $4(m-{i_j})>d$ and hence the product has to vanish because the cohomology of $M$ vanishes above degree $d$, leading to a contradiction.
\end{proof}

\noindent Recall that $p(n)$ is defined to be the number of partitions and the number of those partitions into natural numbers $\le n'$ is $p(n,n')$. \pref{Proposition}{prop:upper-obstruction} yields the following upper bound.

\begin{lem}\label{lem:blockupperbound}
	$\dim\blockfib{M,4m-d}\le p(m) - p(m, m-\floor{\frac{d+1}{4}}) -1$.
\end{lem}

\subsubsection*{Achieving the upper bound}
\noindent We will now show that this upper bound is sharp, i.e. there exists a manifold for which equality holds. 
\begin{dfn}
	A manifold is said to be \emph{$P$-large} if all monomials in rational Pontryagin classes of $TM$ are linearly independent.
\end{dfn}

\noindent If $\tau\colon M\to \bo(d)$ is the classifying map for the tangent bundle $TM$, then $M$ is $P$-large if and only if the induced map $\tau^*\colon H^*(\bo(d);\bbQ)\to H^*(M;\bbQ)$ is injective for $*\le d$. In  the example below we construct a $P$-large manifold, which can be made simply  connected in dimensions $d\equiv2,3(4)$.

\begin{example}\label{ex:large}
	For $n\ge1$ let $M^n_{1},\dots,M^n_{s_n}$ be a basis for $\Omega_{4n}\otimes\bbQ$ with the property that each of those only has one non-trivial monomial Pontryagin number. Since $\Omega_{*}\otimes\bbQ$ is generated by $\cp{2n}$ we may choose $M^n_i$ to be simply connected. Consider the following $d$-dimensional manifold:
	\[M\coloneqq \coprod_{n=1}^{\frac d4}\coprod_{j=1}^{s_n} M^n_j\times S^{d-4n}\]
	This manifold has all possible products of Pontryagin classes and they are all linearly independent, since $H^*(M;\bbQ) = \bigoplus_{n=1}^{d/4}\bigoplus_{j=1}^{s_n} H^*(M^n_j\times S^{d-4n})$ and for every $I=(i_1,\dots, i_s)$ there is a unique $j\in\{1,\dots, s_{|I|}\}$ such that $p_I(TM^{|I|}_{j})\not=0$. Note that for $d\not\equiv 1\;(4)$ every component of $M$ is simply connected. If $d\equiv 2,3\;(4)$, we can even assume $M$ to be simply connected by performing connected sums. If $d\equiv 0,1\;(4)$ this is not possible since Pontryagin products of top degree would then live in the $1$-dimensional space $H^d(M;\bbQ)$ (resp. in the $0$-dimensional space $H^{d-1}(M;\bbQ)$).	
\end{example}

\noindent For $I=(i_1,\dots, i_s)$ let $|I|:=\sum i_j$ and we introduce the short notation 
\[p_I = p_{i_1}\cup\hdots\cup p_{i_s}.\]

\begin{lem}\label{lem:blocklarge}
	Let $M$ be simply connected and $P$-large and let $I =\{i_1,\dots,i_s\}\not=\{m\}$ with $|I|=m$ and $i_j\ge m-\frac {d}4$ for some $j$. Then there exists a normal invariant $\eta\in \calN_\partial(D^{4m-d}\times M)$ with underlying (extended) stable vector bundle $\xi'$ such that
	\[\scpr{p_I(TM\oplus-\xi'), [S^{4m-d}\times M]} \text{ and } \scpr{p_m(TM\oplus-\xi'), [S^{4m-d}\times M]}\]
	are the \emph{only} non-vanishing monomial Pontryagin numbers of $TM\oplus-\xi'$ and $\sigma(\eta) = 0$.
\end{lem}

\begin{proof}
	Since the cup product induces a perfect pairing and all monomials in Pontryagin classes of $M$ are linearly independent, there exist elements $x_J\in H^{d-4|J|}(M)$ for every collection $J=(j_1,\dots,j_t)$ with $|J|<\frac d4$ such that $x_J\cup p_{J'}(TM)=u_M\in H^{d}(M)$ is a generator if $J=J'$ and $x_J\cup p_{J'}(TM)=0$ for $J\not=J'$.

	Without loss of generality let $i_1$ be the biggest element of $I$ and let $a_1$ be the number of elements in $I$ equal to $i_1$. By assumption $i_1 \ge m-\frac d4$ and $p_{I\setminus\{i_1\}}(TM)\not=0$, since $M$ is $P$-large. By \pref{Section}{sec:surgery} there exists a normal invariant $\eta$ such that the underlying stable vector bundle $\xi\to S^k\times M$ has the following Pontryagin classes:
	\begin{align*}
		p_0(-\xi') &= 1\\
		p_{i_1}(-\xi') &= \lambda \cdot u_{4m-d}\times x_{I\setminus \{i_1\}}\\
		p_i(-\xi') &= 0 &\text{ for } i<i_1\\
		p_i(-\xi') &= \lambda\cdot u_{4m-d}\times\sum_{J\colon |J| = m-i} A_J x_J &\text{ for } i>i_1
	\end{align*}
	for a generator $u_{4m-d}\in H^{4m-d}(S^{4m-d})$ and $A_J$ to be determined later. Note that for $|J|=m-i$, the degree of $x_J$ is $d-4|J| = 4i-(4m-d)$. The same computation as the first one in the proof of \pref{Lemma}{lem:realize-p} that 
	\[p_I(TM\oplus -\xi')=\lambda a_1\cdot u_{4m-d}\times u_M \not=0.\]
	It remains to show that we can choose $A_J$ such that all other monomial Pontryagin numbers are trivial. Now let $I'=(i_1',\dots,i_t')$ be different collection with again $|I'|=m$, $i_1'$ the maximum and $a_1'$ the number of elements in $I'$ equal to $i_1'$. Then
	\begin{align*}
		p_{I'} (TM\oplus -\xi') &= \prod_{j=1}^t p_{i_j'}(TM\oplus -\xi') = \prod_{j=1}^t \sum_{a=0}^{i_j'} p_a(-\xi')\cup p_{i_j'-a}(TM)\\
			&=\prod_{j=1}^t \left(p_{i_j'}(TM) + \sum_{a=i_1}^{i_j'} p_a(-\xi')\cup p_{i_j'-a}(TM) \right).
	\end{align*}
	If $i_j'<i_1$ for all $j$, then the sum on the right vanishes and for degree reasons so does the entire expression. If $i_1' = i_1$, then we get
	\begin{align*}
		p_{I'} (TM\oplus -\xi') &= \prod_{j\colon i_j'<i_1'} p_{i_j'}(TM)\cup \prod_{j\colon i_j'=i_1'} \left(p_{i_j'}(TM) + p_{i_j'}(-\xi')\right)\\
			&=\underbrace{p_I(TM)}_{=0} + \prod_{j\colon i_j'<i_1'} p_{i_j'}(TM)\cup\left(\sum_{j\colon i_j'=i_1'}p_{i_1}(-\xi')\cup p_{i_1}(TM)^{a_1'-1}\right)\\
			&=p_{i_1}(-\xi')\cup p_{I'\setminus\{i_1\}}(TM)\\
			&=\lambda  \cdot u_{4m-d}\times x_{I\setminus\{i_1\}} \cup p_{I'\setminus\{i_1\}}(TM)
	\end{align*}
	where the third equality follows from the observation that $p_n(\xi)\cup p_{n'}(\xi)=0$ for $n,n'\ge1$.	By the choice of $x_J$, we have that $x_{I\setminus\{i_1\}} \cup p_{I'\setminus\{i_1\}}(TM)=0$ unless $I=I'$ and therefore $p_{I'}(TM\oplus -\xi') = 0$. It remains to investigate the case $i_1' > i_1$. The strategy is to choose the coefficients $A_{J}$ by downwards induction with respect to $|J|$. Note, that we have already chosen $A_J$ for $|J|\ge m-{i_1}$ to be either $0$ or $1$. Let $J = (j_1,\dots,j_r)$ with $|J|\ge1$ and let us assume that $A_{J'}$ is already chosen for all $J'$ with $|J'|>|J|$. By the choice of the Pontryagin classes of $-\xi'$, this implies that $p_i(-\xi')$ is already determined for all $i<4(m-|J|)\eqqcolon i_J$. If there exists a $j\in J$ such that $j>i_J$, we set $A_J=0$. If not, let $I'\coloneqq \{i_J, J\} :=: \{i_1',\dots, i_t'\}$ and note that by assumption $|I'| = 4m$ and $i_1'$ is the largest entry of $I'$. We again denote the number of indices agreeing with $i_1'$ by $a_1'$. We compute 
	\begin{equation}\label{eq:large1}
	\begin{aligned}
		p_{I'} (TM\oplus -\xi') ={}& \prod_{\ell=1}^t \left(p_{i_\ell'}(TM) + \sum_{a=i_1}^{i_\ell'} p_a(-\xi')\cup p_{i_\ell'-a}(TM) \right)\\
			={}&\prod_{\ell\colon i_\ell'= i_1'}\Bigl(p_{i_\ell'}(TM) + p_{i_\ell'}(-\xi') + \sum_{a=i_1}^{i_\ell'-1} p_a(-\xi')\cup p_{i_\ell'-a}(TM)\Bigr)\\
				&\qquad\cup \prod_{\ell\colon i_\ell'< i_1'}\Bigl(p_{i_\ell'}(TM) + \sum_{a=i_1}^{i_\ell'} p_a(-\xi')\cup p_{i_\ell'-a}(TM)\Bigr)\\
	\end{aligned}
	\end{equation}
	Extracting all summands containing $p_{i_1'}(-\xi)$, we obtain
	\begin{align*}
			(\ref{eq:large1})={}&a_1' p_{i_1'}(-\xi') \cup \underbrace{p_{I'\setminus\{i_1'\}}(TM)}_{ = p_J(TM)}\\
				&+ \underbrace{\left(\prod_{\ell\colon i_\ell'= i_1'}p_{i_\ell'}(TM) + \sum_{a=i_1}^{i_\ell'-1} p_a(-\xi')\cup p_{i_\ell'-a}(TM)\right)\cup \prod_{\ell\colon i_\ell'< i_1'}\Bigl(\cdots\Bigr)}_{\eqqcolon \lambda\cdot B_{I'}} 
	\end{align*}
	Note that the highest index of a Pontryagin class of $-\xi'$ appearing in $B_{I'}$ is strictly smaller than $i_1' = i_J$. Hence, it is only dependent on $A_{J'}$ with $|J'|<|J|$ and by our assumption this summand is already determined. We get
	\begin{align*}
		p_{I'} (TM\oplus -\xi') &= a_1'\lambda \cdot u_{4m-d}\times\sum_{\tilde J\colon |\tilde J| = m-i_1'} A_{\tilde J} \underbrace{x_{\tilde J}\cup p_{J}(TM)}_{=\begin{cases} 0 &\text{ if } \tilde J \not= J\\ u_M &\text{ if } \tilde J = J\end{cases}} + \lambda\cdot B_{I'}\\
			& = \lambda\cdot (a_1' A_{J}\cdot u_{4m-d}\times u_M + B_{I'}).
	\end{align*}
	Therefore we can choose $A_{J}$ for all $J$ with $|J| = m-i_J$ such that $p_{I'}(TM\oplus-\xi')=0$ for all $I'$ with $i_1' = i_J$.\footnote{Note that $I'\not=I$ since $i_1'>i_1$ by assumption.} It remains to specify $A_{\{0\}}$ and hence $p_m(-\xi$). By the same argument as in the proof of \pref{Lemma}{lem:pmin} we can choose $A_{\{0\}}$ such that $\sigma(\eta)=0$ which finishes the proof.
\end{proof}

\begin{cor}
	For a simply connected, $P$-large manifold $M$ of dimension $d$ we have
	\[\dim\blockfib{M,4m-d} = p(m) - p(m, m-\ceil{\frac{d+1}{4}}) -1.\]
\end{cor}
\begin{proof}
	Since the oriented cobordism ring is classified by Pontryagin numbers, the functionals given by evaluating monomial Pontryagin numbers are all linearly independent. Let $\calI_{m,d}\coloneqq\{I=\{i_1,\dots i_s\}\colon I\not=\{m\},\ |I|=m\text{ and } i_j\ge m-d/4\text{ for some $j$}\}$. By \pref{Lemma}{lem:blocklarge}, there exists for every $I\in \calI_{m,d}$ an $M$-block bundle $E_I\to S^k$ such that $\scpr{p_J(E_I), [E]}\not=0$ if and only if $I=J$ for all $J\in\calI_{m,d}$. Therefore, $(E_I)_{I\in\calI_{m,d}}$ is also linearly independent and the claim follows from 
	\[|\calI_{m,d}| = p(m) - p(m, m-\ceil{\frac{d+1}{4}})-1\] 
	and from the upper bound in \pref{Lemma}{lem:blockupperbound}.
\end{proof}

\section{Application to \texorpdfstring{$\calR_C(M)$}{spaces of psc-metrics}}

\subsection{\texorpdfstring{$\hat\calA$}{A-hat}-genus and cross sections with trivial normal bundle}
\noindent For applications to spaces of metrics, we need bundles with non-vanishing $\hat\calA$-genus and a $\Spin$-structure on the total tangent bundle. The following proposition implies \pref{Proposition}{prop:a-hat-multip}.

\begin{prop}\label{prop:bundle}
	Let $M$ be an oriented, simply connected manifold of dimension $d$ and let $M$ that has at least one non-vanishing rational Pontryagin class and let $k$ be in the unblocking range for $M$. If $d+k\equiv0\;(4)$, then there exists a fiber homotopy trivial $M$-bundle $E\to S^k$ that satisfies $\hat\calA(E)\not=0$. 
\end{prop}

\begin{proof}
	By the last part \pref{Lemma}{lem:pmin} together with \pref{Lemma}{lem:unblocking}, there exists an $M$-bundle $E\to S^k$ that has only two non-vanishing monomial Pontryagin numbers, namely $p_m$ and $p_{i_{\min}}\cup p_{m-i_{\min}}$. By \cite[Lemma 2.5]{a-hat-bundles}, this implies that $\hat\calA(E)\not=0$.
\end{proof}

\begin{rem}\label{rem:a-hat}
	\begin{enumerate}
		\item This recovers \cite[Theorem 1.4]{HankeSchickSteimle} and provides an upgrade: the result in loc.cit. is \enquote{based on abstract existence results [and] does not yield an explicit description of the diffeomorphism type of the fiber manifold} \cite[p. 337]{HankeSchickSteimle}. In contrast, our result states, that it is correct for generic manifolds.
		\item By \cite[Proposition 1.9]{HankeSchickSteimle} and \cite[Lemma 2.3]{Wiemeler} a bundle $M\to E\to S^k$ is rationally nullcobordant, if all rational Pontryagin classes vanish or if $\dim(M)<\frac k2$. This implies the necessity for a nonvanishing rational Pontryagin class and a bound on $k$ in terms of $\dim(M)$.
	\end{enumerate}
\end{rem}

\noindent Next, let us investigate, if the bundles we constructed have cross-sections with trivial normal bundle. This is sometimes desirable for applications to positive curvature as it allows for fiber-wise connected sums. We have the following result.

\begin{lem}\label{lem:section}
	If $i_{\min}<d/4$, then the bundle from \pref{Proposition}{prop:bundle} has a cross-section with trivial normal bundle.
\end{lem}
\begin{proof}
	Let $\mathrm{triv}\colon S^k\embeds S^k\times M$ be the trivial section. Since the bundle $E$ from \pref{Proposition}{prop:bundle} is fiber homotopy trivial via $f\colon S^k\times M\simeq E$ we get a section $s\coloneqq f\circ \mathrm{triv}\colon S^k\to E$. We have
	\begin{align*}
		s^*p_n(TE) & = \mathrm{triv}^*\left(\sum_{i=0}^n p_i(TM)\cup p_{n-i}(-\xi')\right)\\
			& = \sum_{i=0}^n \underbrace{\mathrm{triv}^*p_i(TM)}_{=0 \text{ for } i\ge1}\cup \mathrm{triv}^*p_{n-i}(-\xi') = \mathrm{triv}^* p_n(-\xi')
	\end{align*}
	with $\xi'$ as in the proof of \pref{Lemma}{lem:pmin}. Recall, that the only non-vanishing Pontryagin classes of $\xi'$ are $p_{m-i_{\min}}$ and $p_m$ and let $\nu_s$ denote the normal bundle of $s$. Since the rank of this bundle is bigger than $k$, the bundle $\nu_s$ is stable in the sense that it is classified by an element in
	\[
	\pi_k(\bo) = \ko^{-k}(\pt)\cong \begin{cases}
		\bbZ & \text{ for } k\equiv 0\;(4)\\
		\bbZ/2 & \text{ for } k\equiv 1,2\;(8)\\
		0 & \text{ otherwise}.
	\end{cases}
	\]
	Since we are only interested in the problem rationally, it suffices to consider the case $k\equiv0\;(4)$. It follows, that $\nu_s$ is trivial if $p_{k/4}(\nu_s)=0$ and as $p(S^k)=1$, the Pontryagin class $p_{k/4}$ of $\nu_s$ satisfies
	\begin{align*}
		p_{k/4}(\nu_s) = p_{k/4}(s^*TE) = s^*p_{k/4}(TE) = \mathrm{triv}^*p_{k/4}(\xi)=0
	\end{align*}
	since by our assumption $k/4 < \frac{d+k}{4} - i_{\min} = m-i_{\min}$ and $p_{m-i_{\min}}$ and $p_m$ are the only Pontryagin classes of $\xi$.
\end{proof}

\begin{rem}\label{rem:section}
	If $d\not\equiv0\;(4)$, the requirement from the lemma is automatically full-filled. If $d\equiv0\;(4)$ and $i_{\min}=d/4$, then $M$ has only one non-vanishing Pontryagin number, namely $\scpr{p_{d/4}(TM),\ [M]}$. Since all coefficients in the $\hat\calA$-polynomial are nonzero by \cite{BergBerg}, we have $\hat\calA(M) = a\cdot\scpr{p_{d/4}(TM),\ [M]}\not=0$ for some $a\in\bbZ\setminus\{0\}$. If additionally $M$ admits a $\Spin$-structure, then by the Lichnerowicz-formula and the Atiyah--Singer index theorem \cite{atiyahsinger, lichnerowicz}, $M$ does not support a metric of positive scalar curvature. Hence, for a $\Spin$-manifold of positive scalar curvature, we have $i_{\min}<d/4$ and \pref{Lemma}{lem:section} applies.
\end{rem}

\subsection{\texorpdfstring{$\Spin$}{Spin}-structures and positive (scalar) curvature}
Let $M$ be $\Spin$ and let $B\!\diffs(M)$ be the classifying space for $M$-bundles with a $\Spin$-structure on the vertical tangent bundle\footnote{A model for $B\!\diffs(M)$ is given by 
\[ B\!\diffs(M) \coloneqq\{(N,\hat \ell_n), M\cong N\subset\bbR^{\infty}, \hat\ell_N\in\mathrm{Bun}(TN, \theta^*U_d)\}\]
for $\theta\colon B\!\Spin(d)\to BSO(d)$ the $2$-connected cover, $U_d\to BSO(d)$ the universal oriented vector bundle and $\mathrm{Bun}(\_,\_)$ the space of bundle maps.}. By \cite[Lemma 3.3.6]{ebert_thesis} the homotopy fiber of the forgetful map $B\!\diffs(M)\to B\!\diff(M)$ is a $K(\bbZ/2,1)$ if $M$ is simply connected. Therefore the induced map 
\[\pi_n(B\!\diffs(M))\otimes\bbQ\too\pi_n(B\!\diff(M))\otimes\bbQ\]
is an isomorphism and we may assume without loss of generality that the bundles from \pref{Section}{sec:prescribing} carry a $\Spin$-structure on the vertical tangent bundle and hence on the total space, provided that $M$ admits one. \pref{Theorem}{main:curvature} then follows from \pref{Proposition}{prop:bundle} by a standard argument that goes back to Hitchin \cite{hitchin_spinors} (see \cite[Remark 1.5]{HankeSchickSteimle} or \cite[Proposition 3.7]{a-hat-bundles}). 

Applying \pref{Theorem}{main:curvature}, we get the following classification for the push-forward action on metrics of positive scalar curvature which uses rigidity results from \cite{erw_psc3} and \cite{actionofmcg}. 

\begin{cor}\label{cor:psc}
	Let $M$ be a 2-connected, $d$-dimensional $\Spin$-manifold of positive scalar curvature and let $k$ be in the unblocking range for $M$. 
	\begin{enumerate}
		\item Then the orbit map $\pi_{k-1}\diff(M,D)\to \pi_{k-1}(\calR_\psc(M))$ factors through a finite group if and only if $(d+k)$ is not divisible by four or all rational Pontryagin classes of $M$ vanish.
		\item If $k=1$, then the same holds for $\diff(M)$ instead of $\diff(M,D)$.
	\end{enumerate}
\end{cor}

\begin{proof}%[Proof of \pref{Corollary}{cor:psc}]
	The orbit maps 
	\[\pi_{k-1}(\diff(M,D))\to \pi_{k-1}(\calR_\psc(M))\]
	factor through finite groups if all Pontryagin classes of $W$ vanish by \cite[Theorem F]{erw_psc3}). Furthermore, $\pi_{0}(\diff(M))\to \pi_{0}(\calR_\psc(M))$ factors through a finite group by \cite[Theorem A]{actionofmcg}. The rest follows from \pref{Theorem}{main:curvature}.
\end{proof}

\noindent\pref{Theorem}{main:curvature} also allows to recover the main result from \cite{HankeSchickSteimle} (loc.cit. Theorem 1.1 a)) and is actually slightly more precise on the dimension restriction.

\begin{cor}
	Let $k\ge1$ and let $N$ be a $\Spin$-manifold of positive scalar curvature such that $d+k\equiv0\;(4)$ and $k$ is in the unblocking range for $M$. Then the group $\pi_{k-1}(\calR_\psc(N))$ contains an element of infinite order (resp. is infinite if $k=1$).
\end{cor}

\begin{proof}
Let $K$ be a $K3$-surface. Then for $n\coloneqq d-4\ge2$, the manifold $K\times S^n$ satisfies the hypothesis of \pref{Theorem}{main:curvature} and there is a $K\times S^n$-bundle $E\to S^k$ that has non-vanishing $\hat\calA$-genus and admits a cross section with trivial normal bundle. Gluing in the trivial $N\setminus D^d$-bundle along this cross section yields a $N\#(K\times S^{d-4})$-bundle over $S^k$ with non-vanishing $\hat\calA$-genus. Hence the group $\pi_{k-1}(\calR_\psc(N\#(K\times S^{d-4})))$ contains an element of infinite order. Since $N$ is cobordant to $N\#(K\times S^{d-4})$ in $\Omega_{\Spin}^d(B\pi_1(N))$, the corresponding spaces of positive scalar curvature metrics are homotopy equivalent by \cite[Theorem 1.5]{ebertfrenck}.
\end{proof}

\begin{rem}
	A more general result without any dimension restriction has been proven by Botvinnik--Ebert--Randal-Williams \cite{berw}. The methods from loc.cit. are however not constructive and do not give a way to decide if the obtained elements arise from the orbit of the action $\diff(M)\actson\calR_\psc(M)$. Furthermore it is unclear if those elements originate from the spaces $\calR_\prc(M)$ or $\calR_\psec(M)$.
\end{rem}

\appendix
%!TEX root = main.tex
%\newpage
\vspace{33pt}
\section{Rationally fibering a cobordism class over a sphere \\ with Jens Reinhold}\label{sec:appendix}

\noindent This appendix promotes the problem of studying the ideal of oriented cobordism classes that have a representative fibering over a sphere of fixed dimension. Such a class also fibers over any manifold of smaller dimension, see \pref{Proposition}{prop:changeofdimension}. An answer thus has consequences for other bases, too. We are only interested in the rational version. It turns out that the results from the present paper can be used to say something new about this problem, which has been solved (even integrally) for dimensions at most $4$ some time ago: in this case the rational answer is that a cobordism class fibers over $S^k$ for $k \leq 4$ if and only if its signature vanishes \cite{Burdick, Neumann, Kahn1, Kahn2}. A variant of the analogous problem without orientations was originally introduced by Connor and Floyd \cite{ConnorFloyd}. We describe a construction that goes beyond the way bundles arise in the preceding paper. Unfortunately it seems as if even both ideas combined are not sufficient to solve the problem completely unless $k \leq 8$. We first outline a more concrete version of the problem. Let $\Omega_{\ast}$ denote the (graded) oriented cobordism ring. 

\begin{dfn} 
\begin{enumerate}
	\item An oriented cobordism class $\alpha \in \Omega_{\ast}$ is said to \emph{fiber over} a manifold $B$ if there is an oriented smooth fiber bundle $M \to E^d \to B$ such that $[E] = \alpha$.
	\item For $k\ge1$, let $A^k_{\ast} \subset \Omega_{\ast}$ denote the graded subgroup spanned by cobordism classes that fiber over  $S^k$.
	\item For given $k,m\ge 1$, define $c^k(m)\in\mathbb Z$ by
		\[c^k(m) \coloneqq  \text{dim}_{\mathbb Q}(\Omega_{4m}\otimes\bfQ) - \text{dim}_{\mathbb Q}(A^{k}_{4m} \otimes\bfQ ) - 1\]
\end{enumerate} 
\end{dfn}

\noindent Forming disjoint unions and products, we see that $A^k_{\ast}$ is an ideal in $\Omega_{\ast}$ and we may ask what these ideals are depending on $k$. As the signature of any manifold that fibers over a sphere vanishes, the two maps 
\[A^k_{\ast} \otimes\mathbb Q \hookrightarrow \Omega_{\ast}\otimes\mathbb Q \xrightarrow{\mathcal L_{\ast}} \bfQ\]

\noindent compose to $0$. As there exist manifolds of non-zero signature in any dimension divisible by $4$, this implies $c^k(m) \geq 0$. We may ask if the above sequence is exact in sufficiently high degrees, or equivalently (see part (ii)):

\begin{problem} \label{MainQuestion}
\begin{enumerate}
\item[(i)] Describe the ideals $A^k_{\ast} \subset \Omega_{\ast}$ for all values of $k$.\item[(ii)]  Is $c^k(m) = 0$ for fixed $k$ and sufficiently large $m$?
\end{enumerate}
\end{problem}

\noindent We will see below that we (at least) need to restrict to degrees $m \ge k/2$ for (ii) to be true: there are more constraints than the vanishing of the signature in lower degrees, see \pref{Proposition}{prop:constraintsinlowdegrees}. The following is our contribution towards an answer to \pref{Problem}{MainQuestion}.

\begin{thm}\label{MainThmAppendix} Let $k \ge 1$ be fixed.
\begin{enumerate}
\item[(i)] We have $c^k(m) = \dim\Omega_{4m}-1$ for $m < \frac{3k}8$, and $c^k(m) \ge 1$ for $m < \frac k2$.
\item[(ii)] For $5 \le k \le 8$ we have $c^k(m) = 0$ for $m \ge k$.
\item[(iii)] For $9\le k \le 12$ we have $c^k(m)\le 6$ in degrees $m \ge k$.
\end{enumerate}
\end{thm}

\noindent Regarding the last item we note that from computer-aided calculations we know that $c^{k}(m) = 0$ for $k \leq 12$ and $m \leq 500$, see \pref{Remark}{rem:computer-calculations}.
 
We prove \pref{Theorem}{MainThmAppendix} below. Before doing so, let us elaborate on the consequences of the preceding paper regarding a partial answer to \pref{Question}{MainQuestion} for bigger values of $k$: sharpness of the upper bound from \pref{Theorem}{main:upperbound} (see also \pref{Remark}{rem:upperbound}) can be reformulated as $c^k(m)\le p(m, \floor{(k-1)/4})$. Note that $p(n,\ell) = p(n-\ell,\ell) + p(n,\ell-1)$. Using $p(n,1) = 1$ a simple induction shows that $p(n,\ell) = \calO(n^{\ell-1})$, which yields the following consequence of \pref{Theorem}{main:upperbound}.

\begin{cor}\label{cor:mainrephrased}
	For $k\ge1$ and $m > k$, we have 
		\[c^k(m) \le p(m,\floor{(k-1)/4}) = \calO(m^{\floor{(k-1)/4}-1}).\]
\end{cor}

\noindent The rest of the appendix is devoted to proving  \pref{Theorem}{MainThmAppendix}.

\subsection*{Elementary observations}
We first collect some elementary facts about the ideals $A^k_{\ast} \subset \Omega_{\ast}$.
\begin{prop} \label{prop:changeofdimension}
A cobordism class that fibers over $S^k$ fibers over any $(k-1)$-manifold $B$.
\end{prop}
\begin{proof} Cutting out a $(k+1)$-disk from a nullbordism of $B \times S^1$, we see that $S^k$ and $B \times S^1$ can be joined by a connected oriented cobordism $W$. Applying obstruction theory to a relative CW-decomposition of $(W,S^{k})$ and using that the obstructions lie in $H^j(W,S^k;\pi_{j-1}(S^{k})) = 0$ for all $j$, we see that $W$ retracts onto $S^k$. For any smooth bundle $M \to E \to S^k$, we can thus extend the classifying map $S^k \to B\text{Diff}(M)$ to $W$. Restricting this extension to the other end of the cobordism gives a bundle $M \to E' \to B \times S^1$ with $[E'] = [E] \in \Omega_{4m}$. Since $E'$ clearly also fibers over $B$, this finishes the proof.
\end{proof}

\begin{rem}\label{rem:changeofdimension}
Replacing $\bdiff(M)$ by $\haut(M)/\diff(M)$ the same proof yields that a fiber homotopy trivial $M$-bundle over $S^k$ is cobordant to a fiber homotopy trivial $M\times S^1$-bundle over $B$.
\end{rem}

\noindent Proposition \ref{prop:changeofdimension} implies that $(A^k_{\ast})_{k\ge1}$ forms a decreasing chain of ideals of $\Omega_{\ast}$. We next prove part (i) of \pref{Theorem}{MainThmAppendix}.
%\begin{prop}
%[{c.f.~\cite[Lemma.~2.3]{Wiemeler}}]\label{prop:wiemeler}
%We have $A^k_{\ast} \otimes \bfQ = 0$ in degrees $\ast < \frac{3}{8}k$.
%\end{prop}
\begin{proof} (c.f.~\cite[Lemma.~2.3]{Wiemeler}) \label{prop:wiemeler} We need to show that for any smooth bundle $\pi\colon E^{4m}\to S^k$ with $d \coloneqq (4m-k)$-dimensional fiber $M$ such that $4m < \frac{3}{2}k$, we have $[E] = 0 \in \Omega_{4m}\otimes\bfQ$. 

Since the tangent bundle $TE$ is stably isomorpic to the vertical tangent bundle $T_{\pi}E$ whose dimension is $d$, we deduce that only Pontryagin classes $p_i$ with $i \le 2d$ can be non-zero. 

Analyzing the Serre spectral sequence of the fibration $M \to E \to S^k$ yields that $E$ has no cohomology in degrees $d < \ast < k$. The assumption implies that $k > 2d$, hence we deduce that all monomials in Pontryagin classes of $E$ in degrees at least $k$ vanish. In particular, all composite Pontryagin numbers of $E$ are zero. Together with the result on the vanishing signature that we recalled above, we deduce that $E$ is rationally nullbordant.

The second part of the assertion immediately follows from \pref{Proposition}{prop:constraintsinlowdegrees} that we state and prove next.
\end{proof}

%A weaker bound on the dimension of the base sphere yields only the vanishing of some Pontryagin classes:

\begin{prop}
\label{prop:constraintsinlowdegrees}
	Let $\pi\colon E^{4m}\to S^k$ be a fiber bundle with $4m<2k$. Then $p_{i}(TE)=0 $ for all $i > 2m-\frac{k}{2}$.  (The inequality ensures this number is smaller than $m$.)
\end{prop}
\begin{proof}
	Let $T_{\pi}E$ be the vertical tangent bundle of $\pi$. We have
	\[p_{i}(TE) = p_{i}(T_{\pi}E \oplus TS^k) = p_{i}(T_{\pi}E).\]
	Now $\rk(T_\pi E)= 4m-k$ and so any $p_{i}(T_\pi E)$ with $ i \ge \frac{1}{2}(4m-k)$ vanishes.
\end{proof}
%%%
%%
%
\subsection*{Bundles that are trivial as fibrations}
Note that constructions arising from block bundles yield bundles that are trivial as fibrations. For such bundles the following vanishing result holds, which implies that the analogue of \pref{Problem}{MainQuestion} (ii) for fiber-homotopically trivial bundles has a negative answer.

\begin{prop} 
\label{prop:fiberhomotopicallytrivialbundles}
For a fiber-homotopically trivial bundle $M \to E^{4m} \to B$ whose base space $B$ is $4\ell$-connected and $p \in \oH(BSO(4m);\bfQ)$ a monomial in Pontryagin classes $p_i$ with $i \leq \ell$, the Pontryagin number $p(E)$ vanishes.
\end{prop}
\begin{proof} From the assumption that the bundle is trivial as a fibration, we deduce that $E \sim M \times B$. In particular, we get a retraction $E \to M$ for the inclusion of a fiber. But $B$ is $4\ell$-connected, hence all Pontryagin classes $p_i$ with $i \leq \ell$ pull back along this map. Since $\oH^{4m}(M) = 0$, this implies the assertion.
\end{proof}
%%%
%%
%
\subsection*{Constructing a bundle that is non-trivial as a fibration} In this section, we construct for any $m \ge 1$ a bundle $\mathbb C P^{m} \to E \to S^{2m}$ so that $p_1^m(E) \neq 0$ if $m \ge 3$. We have seen in Proposition \ref{prop:fiberhomotopicallytrivialbundles} that the latter is not possible for bundles that are trivial as fibrations.

\begin{construction}
\label{construction_bundle}
Let $m \geq 1$. We construct a smooth  $\mathbb C P^{m}$-bundle over  $S^{2m}$ as follows. The topological group $\GL_{m}(\mathbb C)$ acts on \[\mathbb C P^m = \{[z_0:z_1: \ldots : z_m] \ | \ z_i \in \mathbb C \text{ not all } 0\}\] by acting linearly on the last $m$ projective coordinates. This action fixes the point $\ast \coloneqq [1:0:\ldots:0]$ and induces a map 
\[ 
B\!\GL_m(\mathbb C) \to B\!\Diff(\mathbb C P^m,\ast).
\]
The action of a differential on the tangent space of this fixed point produces a map
\[ 
B\!\Diff(\mathbb C P^m,\ast) \to B\!\GL_{2m}(\Reals),
\]
and it is evident that the composition of these two maps is the canonical map $B\!\GL_m(\mathbb C) \to B\!\GL_{2m}(\Reals)$ induced from seeing $\mathbb C$ as a $2$-dimensional real vector space. We now choose a complex $m$-dimensional vector bundle over $S^{2m}$, classified by a map 
$
S^{2m} \to B\!\GL_m(\mathbb C),
$
whose underlying $2m$-dimensional real vector bundle $\xi$ has a non-zero Euler number. 
When composed with the previous map, we obtain a map classifying a smooth bundle $\mathbb C P^m \to E \to S^{2m}$. %with a section.
\end{construction}

\begin{prop}\label{prop:p_1^n}
	If $m\ge 3$, then the bundle from Construction \ref{construction_bundle} satisfies  and we have:
	\begin{align*}
		p_1(E)^m &\neq 0\\
		p_i(E) &= \binom{m+1}{i}\cdot\left(\frac{1}{m+1}\right)^i\cdot p_1(E)^i\quad\text{ for } 2i<m
	\end{align*}
\end{prop}
\begin{proof} Choose a generator $\alpha \in \oH^2(\mathbb C P^m)$. Then there exists a unique class $\beta \in \oH^2(E)$ that pulls back to $\alpha$ along the inclusion $j \colon \mathbb C P^m \to E$ of the fiber. From the Serre spectral sequence of the bundle  $\mathbb C P^m \to E \xrightarrow{\pi} S^{2m}$ which collapses since the $E^2$ page is supported in even degrees, we see that $\beta^m$ is Poincar\'e dual to a non-zero multiple of $\pi^{\ast} [S^{2m}]$, where $[S^{2m}] \in \oH_{2m}(S^{2m})$ denotes the fundamental class. We thus get that $\beta^{2m} \in \oH^{4m}(E)$ is Poincar\'e dual to the Euler number of $\xi$, and hence non-zero. Since $S^{2m}$ has a trivial tangent bundle, we deduce $j^{\ast} p_1(E) = p_1(\mathbb C P^m) = (m+1)\alpha^2$, hence $p_1(E) = (m+1)\beta$. Hence indeed $p_1^m(E) = (m+1)^{2m}\beta^{2m} \neq 0$.

The class $p_i(E)$ is computed as follows:
\begin{align*}
	j^*p_i(E) &= p_i(\cp{2m}) = \binom{m+1}{i} a^{2i} = \binom{m+1}{i}\left(\frac{1}{m+1}\right)^i ((m+1)a^{2})^i\\
		&= \binom{m+1}{i}\left(\frac{1}{m+1}\right)^i p_1(\cp{2m})^i = \binom{m+1}{i}\left(\frac{1}{m+1}\right)^ij^*p_1(E)^i
\end{align*}
Since $j^*$ is injective in degrees smaller than $2m-1$, the claim follows.
\end{proof}

\begin{proof}[Proof of \pref{Theorem}{thm:main} (iii)]
	Consider the bundle $\cp{k}\to E\to S^{2k}$ constructed  in \pref{Proposition}{prop:p_1^n}. Taking the product with $\cp{2\ell}$ we obtain a fiber bundle $\tilde E\coloneqq E\times \cp{2\ell}\to S^{2k}$ with fiber $\cp{2\ell}\times \cp{n}$. We have
	\begin{align*}
		p_1(T\tilde E)^{k+\ell} &= (p_1(T\cp{n})\times 1 + 1\times p_1(TE))^{k+\ell}\\
		&= \binom{k+\ell}{k} p_1(T\cp{n})^\ell\times  p_1(TE)^{k}\not=0
	\end{align*}
\end{proof}

\noindent We next prove part (ii) of \pref{Theorem}{MainThmAppendix}. 

\begin{proof} Assume that $5 \le k \le 8$ and $m > k$. Then \pref{Corollary}{cor:mainrephrased} says that $c^k(m) \le 1$. We want to improve this to $c^k(m) = 0$. To do so, observe that the proof of \pref{Corollary}{cor:mainrephrased}, which was simply a reformulation of (the sharpness of the upper bound of) \pref{Theorem}{main:upperbound}, only involved bundles that are trivial as fibrations. For any such bundle, we know from \pref{Proposition}{prop:fiberhomotopicallytrivialbundles} that $p_1^m(E) = 0$. However, the bundle arising from \pref{Construction}{construction_bundle} satisfies $p_1^m(E) \neq 0$, we have thus found another element in $A^k_{4m}$ which is not fiber homotopy trivial and so we have finished the proof.
\end{proof}

\noindent Finally, we prove part (iii) of \pref{Theorem}{MainThmAppendix}.

\begin{proof}
For $i=1,\dots, m$, let $E_i\to S^{2i}$ denote the $\cp i$-bundle from \pref{Construction}{construction_bundle}. First note that for $i\ge5$, we have
\begin{align*}
	p_2(E_i) = (j^*)^{-1}p_2(\cp{2i}) = \frac i2 (j^*)^{-1}p_1(\cp{2i})^2 = \frac i2 p_1(E_i)^2
\end{align*}
Next, let $Q_i$ be a manifold of dimension $4(m-i)$ such that $p_1^{m-i}(Q_i)\not=0$ is the only non-vanishing Pontryagin number and let $X_i\coloneqq E_i\times Q_i$. Note, that $X_i$ is a fiber bundle with fiber $\cp{2i}\times Q_i$ over $S^{2i}$. We consider the following matrix:
\[B^m\coloneqq\Bigl(p_2^j(X_i)\cdot p_1^{m-2j}(X_i)\Bigr)_{\begin{mysubarray}&i&=6\dots m\\ &j&=0\dots\floor{\frac{m}{2}}\end{mysubarray}}\]
If $\rk(B^m) \ge \floor{\frac m2} +1 - a$, then $c^k(m) \le a$ for $k\le 12$. 
\begin{align*}
	p_2^j(X_i)&\cdot p_1^{m-2j}(X_i) = p_2^j(E_i\times Q_i)\cdot p_1^{m-2j}(E_i\times Q_i)\\
		={}& \Bigl(p_2(E_i) + p_1(E_i)p_1(Q_i) + p_2(Q_i)\Bigr)^j \cdot \Bigl(p_1(E_i) + p_1(Q_i)\Bigr)^{m-2j}
\end{align*}
By our choice of $Q_i$, any product containing $p_2(Q_i)$ will vanish and therefore we can go on with our computation.
\begin{align*}
		={}& \Bigl(p_2(E_i) + p_1(E_i)p_1(Q_i)\Bigr)^j \cdot \Bigl(p_1(E_i) + p_1(Q_i)\Bigr)^{m-2j} \\
		={}& p_1(Q_i)^{m-i}\cdot \sum_{n=0}^{\floor{\frac i2}} p_2(E_i)^n\cdot p_1(E_i)^{j-n}\cdot p_1(E_i)^{i-(j+n)}\cdot \binom{j}{n}\binom{m-2j}{i-(j+n)}\\
		={}&p_1(Q_i)^{m-i}\cdot\sum_{n=0}^{\floor{\frac i2}} \left(\frac{i}{2}\right)^{n}\cdot \binom{j}{n}\binom{m-2j}{i-(j+n)} p_1(E_i)^{2n}\cdot p_1(E_i)^{j-n}\cdot p_1(E_i)^{i-(j+n)}\\
		={}& \underbrace{\frac{p_1(Q_i)^{m-i}\cdot p_{1}(E_i)^i}{2^m}}_{\not=0}\cdot  \sum_{n=0}^{\floor{\frac i2}} 2^{m-n}\ i^{n}\ \binom{j}{n}\binom{m-2j}{i-(j+n)} \\
\end{align*} 
and hence it suffices to compute or estimate the rank of the following matrix:
\[A^m = (A^m_{ij})_{\begin{mysubarray}&i&=6\dots m\\ &j&=0\dots\floor{\frac{m}{2}}\end{mysubarray}} \coloneqq \left(\sum_{n=0}^{\floor{\frac i2}} 2^{m-n}\ i^{n}\ \binom{j}{n}\binom{m-2j}{i-(j+n)}\right)_{\begin{mysubarray}&i&=6\dots m \\ &j&=0\dots\floor{\frac{m}{2}}\end{mysubarray}}\]
Note that for $j>i$ we have $\binom{m-2j}{i-(j+n)}=0$ for all $n\ge0$ and hence $A_{ij}=0$. Therefore the matrix $A$ has the following form, where the asterisks represent non-zero entries.
\[A = \begin{pmatrix}
			\huge * &\dots & * & 0 & \dots & 0\\
			 & & & \ddots & \ddots & \vdots\\
			\vdots & & &  & * & 0\\
			 & & & & & * \\
			\vdots & & & & & \vdots \\
			* & & \dots & \dots  & &  *
	\end{pmatrix}\]
	In the first row, there are $7$ non-zero entries, so the rank of $A$ is at least $\floor{\frac{m}{2}}-5$.
\end{proof}

\begin{rem}\label{rem:computer-calculations}
	Computer calculations, for which we thank Marek Kaluba, have shown that the matrix $A^m$ and hence the matrix $B^m$ as well have rank equal to $\floor{\frac{m}{2}}+1$ for $m\le 500$. This implies that $c^k(m)=0$ for $k\le12$ and $m\le 500$ which can be rephrased in the following way: For every $k\le12$ and any oriented manifold $M$ of dimension at most $2000$ with vanishing signature, there exists a $\lambda\in\bbN$ such that the $\lambda$-fold connected sum of $M$ with itself is cobordant to a fiber bundle $E\to S^k$.
\end{rem}

\vspace{15pt}
{\footnotesize 
\noindent 
Jens Reinhold was supported by the DFG (German Research Foundation) -- SFB 1442 427320536, Geometry: Deformations and Rigidity, as well as under Germany's Excellence Strategy EXC 2044 390685587, Mathematics M\"unster: Dynamics-Geometry-Structure.
\newline \newline
\noindent 
\address{Mathematisches Institut, Einsteinstr.~62, 48149 M{\"u}nster, Germany}
\newline
\noindent 
\texttt{jens.reinhold@uni-muenster.de \\ jens.reinhold@posteo.de}
}

\begin{changemargin}{-2cm}{-2cm}
\bigskip
\printbibliography
\bigskip
\end{changemargin}

\end{document}